\documentclass[11pt]{article}
\usepackage{amsmath, amssymb, amsthm}

\date{}
\title{\vspace{-0.9cm}The minimum number of disjoint pairs in set systems and related problems}
\author{
Shagnik Das \thanks{Department of Mathematics, UCLA, Los Angeles, CA, 90095. Email: shagnik@ucla.edu.}
\and
Wenying Gan \thanks{Department of Mathematics, UCLA, Los
Angeles, CA, 90095. Email: wgan@math.ucla.edu.}
\and
Benny Sudakov \thanks{Department of Mathematics, UCLA, Los Angeles, CA 90095.
Email: bsudakov@math.ucla.edu. Research supported in part by NSF grant DMS-1101185, by AFOSR MURI grant FA9550-10-1-0569 and by a USA-Israel BSF grant. }
}

\theoremstyle{plain}
\newtheorem{THM}{Theorem}[section]
\newtheorem{PROP}[THM]{Proposition}
\newtheorem{LEMMA}[THM]{Lemma}

\newtheorem{COR}[THM]{Corollary}
\newtheorem{CLAIM}{Claim}
\newtheorem{CONJ}[THM]{Conjecture}
\newtheorem{QUE}[THM]{Question}

\theoremstyle{definition}

\newcommand{\floor}[1]{\left\lfloor #1 \right\rfloor}

\newcommand{\card}[1]{\left| #1 \right|}
\newcommand{\cA}{\mathcal{A}}

\newcommand{\cF}{\mathcal{F}}
\newcommand{\cG}{\mathcal{G}}

\newcommand{\cL}{\mathcal{L}}
\newcommand{\cX}{\mathcal{X}}
\newcommand{\disj}{\mathrm{dp}}
\newcommand{\qdisj}{\mathrm{dp}^{(q)}}
\newcommand{\tdisj}{\mathrm{dp}_t}
\newcommand{\tint}{\mathrm{int}_t}

\begin{document}
\maketitle

\begin{abstract}
Let $\cF$ be a set system on $[n]$ with all sets having $k$ elements and every pair of sets intersecting. The celebrated theorem of Erd\H{o}s-Ko-Rado from 1961 says that any such system has size at most $\binom{n-1}{k-1}$. A natural question, which was asked by Ahlswede in 1980, is how many disjoint pairs must appear in a set system of larger size. Except for the case $k=2$, solved by Ahlswede and Katona, this problem has remained open for the last three decades.

In this paper, we determine the minimum number of disjoint pairs in small $k$-uniform families, thus confirming a conjecture of Bollob\'as and Leader in these cases. Moreover, we obtain similar results for two well-known extensions of the Erd\H{o}s-Ko-Rado theorem, determining the minimum number of matchings of size $q$ and the minimum number of $t$-disjoint pairs that appear in set systems larger than the corresponding extremal bounds. In the latter case, this provides a partial solution to a problem of Kleitman and West.
\end{abstract}

\section{Introduction} \label{sec:intro}

A set system $\cF$ is said to be \emph{intersecting} if $F_1 \cap F_2 \neq \emptyset$ for all $F_1, F_2 \in \cF$.  The Erd\H{o}s-Ko-Rado Theorem is a classical result in extremal set theory, determining how large an intersecting $k$-uniform set system can be.  This gives rise to the natural question of how many disjoint pairs must appear in larger set systems.

We consider this problem, first asked by Ahslwede in 1980.  Given a $k$-uniform set system $\cF$ on $[n]$ with $s$ sets, how many disjoint pairs must $\cF$ contain?  We denote the minimum by $\disj(n,k,s)$, and determine its value for a range of system sizes $s$, thus confirming a conjecture of Bollob\'as and Leader in these cases.  This results in a quantitative strengthening of the Erd\H{o}s-Ko-Rado Theorem.  We also provide similar results regarding some well-known extensions of the Erd\H{o}s-Ko-Rado Theorem, which in particular allow us to partially resolve a problem of Kleitman and West.

We now discuss the Erd\H{o}s-Ko-Rado Theorem and the history of this problem in greater detail, before presenting our new results.

\subsection{Intersecting systems}

Extremal set theory is one of the most rapidly developing areas of combinatorics, having enjoyed tremendous growth in recent years.  The field is built on the study of very robust structures, which allow for numerous applications to other branches of mathematics and computer science, including discrete geometry, functional analysis, number theory and complexity.

One such structure that has attracted a great deal of attention over the years is the intersecting set system; that is, a collection $\cF$ of subsets of $[n]$ that is pairwise-intersecting.  The most fundamental question one may ask is how large such a system can be.  Observe that we must have $\card{\cF} \le 2^{n-1}$, since for every set $F \subset [n]$, we can have at most one of $F$ and $[n] \setminus F$ in $\cF$.  This bound is easily seen to be tight, and there are in fact numerous extremal systems.  For example, one could take the set system consisting of all sets containing some fixed $i \in [n]$.  Another construction is to take all sets $F \subset [n]$ of size $\card{F} > \frac{n}{2}$.  If $n$ is odd, this consists of precisely $2^{n-1}$ sets.  If $n$ is even, then we must add an intersecting system of sets of size $\frac{n}{2}$; for instance, $\{ F \subset [n]: \card{F} = \frac{n}{2}, 1 \in F \}$ would suffice.

\medskip

In some sense, having large sets makes it easier for the system to be intersecting.  This leads to the classic theorem of Erd\H{o}s-Ko-Rado \cite{ekr61}, a central result in extremal set theory, which bounds the size of an intersecting set system with all sets restricted to have size $k$.

\begin{THM}[Erd\H{o}s-Ko-Rado \cite{ekr61}, 1961] \label{thm:ekr}
If $n \ge 2k$, and $\cF \subset \binom{[n]}{k}$ is an intersecting set system, then $\card{\cF} \le \binom{n-1}{k-1}$.
\end{THM}

This is again tight, as we may take all sets containing some fixed element $i \in [n]$, a system we call a \emph{(full) star with center $i$}.

\medskip

As is befitting of such an important theorem, there have been numerous extensions to many different settings, some of which are discussed in Anderson's book \cite{anderson}.  We are particular interested in two, namely $t$-intersecting systems and $q$-matching-free systems.

\medskip

A set system $\cF$ is said to be $\emph{$t$-intersecting}$ if $\card{F_1 \cap F_2} \ge t$ for all $F_1, F_2 \in \cF$.  When $t = 1$, we simply have an intersecting system.  A natural construction of a $t$-intersecting system is to fix some $t$-set $X \in \binom{[n]}{t}$, and take all $k$ sets containing $X$; we call this a \emph{(full) $t$-star with center $X$}.  In their original paper, Erd\H{o}s-Ko-Rado showed that, provided $n$ was sufficiently large, this was best possible.

\begin{THM}[Erd\H{o}s-Ko-Rado \cite{ekr61}, 1961] \label{thm:ekrt}
If $n \ge n_0(k,t)$, and $\cF \subset \binom{[n]}{k}$ is a $t$-intersecting set system, then $\card{\cF} \le \binom{n-t}{k-t}$.
\end{THM}

There was much work done on determining the correct value of $n_0(k,t)$, and how large $t$-intersecting systems can be when $n$ is small.  This problem was completely resolved by the celebrated Complete Intersection Theorem of Ahlswede and Khachatrian \cite{ahlkha97} in 1997.

\medskip

The second extension we shall consider concerns matchings.  A \emph{$q$-matching} is a collection of $q$ pairwise-disjoint sets.  A set system is therefore intersecting if and only if it does not contain a $2$-matching.  As an extension of the Erd\H{o}s-Ko-Rado theorem, Erd\H{o}s asked how large a $q$-matching-free $k$-uniform set system could be, and in \cite{erdos65} showed that when $n$ is large, the best construction consists of taking all sets meeting $[q-1]$.  He further conjectured what the solution should be for small $n$, and this remains an open problem of great interest.  For recent results on this conjecture, see, e.g., \cite{frankl13, frr11, hls12, lm12}.

\subsection{Beyond the thresholds} \label{sec:beyond}

The preceding results are all examples of the typical extremal problem, which asks how large a structure can be without containing a forbidden configuration.  In this paper, we study their Erd\H{o}s-Rademacher variants, a name we now explain.

Arguably the most well-known result in extremal combinatorics is a theorem of Mantel \cite{mantel07} from 1907, which states that an $n$-vertex triangle-free graph can have at most $\floor{ \frac{n^2}{4}}$ edges.  In an unpublished result, Rademacher strengthened this theorem by showing that any graph with $\floor{\frac{n^2}{4}} + 1$ edges must contain at least $\floor{\frac{n}{2}}$ triangles.  In \cite{erdos62a} and \cite{erdos62b}, Erd\H{o}s extended this first to graphs with a linear number of extra edges, and then to cliques larger than triangles.  More generally, for any extremal problem, the corresponding Erd\H{o}s-Rademacher problem asks how many copies of the forbidden configuration must appear in a structure larger than the extremal bound.

\medskip

In the context of intersecting systems, the Erd\H{o}s-Rademacher question was first investigated by Frankl \cite{frankl77} and, independently, Ahlswede \cite{ahlswede80} some forty years ago, who showed that the number of disjoint pairs of sets in a set system is minimized by taking the sets to be as large as possible.
\begin{THM}[Frankl \cite{frankl77}, 1977; Ahlswede \cite{ahlswede80}, 1980] \label{thm:erintersect}
If $\sum_{i=k+1}^n \binom{n}{i} \le s \le \sum_{i=k}^n \binom{n}{i}$, then the minimum number of disjoint pairs in a set system of size $s$ is attained by some system $\cF$ with $\cup_{i > k} \binom{[n]}{i} \subseteq \cF \subseteq \cup_{i \ge k} \binom{[n]}{i}$.
\end{THM}

Note that while this theorem provides the large-scale structure of extremal systems, it does not determine exactly which systems are optimal.  Since we have $\cup_{i > k} \binom{[n]}{i} \subset \cF$, each set of size $k$ contributes the same number of disjoint pairs with larger sets.  Hence the total number of disjoint pairs is minimized by minimizing the number of disjoint pairs between the sets of size $k$, a problem raised by Ahlswede.

\begin{QUE}[Ahlswede \cite{ahlswede80}, 1980] \label{que:ahlswede}
Given $0 \le s \le \binom{n}{k}$, which $k$-uniform set systems $\cF \subset \binom{[n]}{k}$ with $\card{\cF} = s$ minimize the number of disjoint pairs?
\end{QUE}

By the Erd\H{o}s-Ko-Rado Theorem, we know that when $s \le \binom{n-1}{k-1}$, we need not have any disjoint pairs, while for $s > \binom{n-1}{k-1}$, there must be at least one disjoint pair.  This question can thus be thought of as the Erd\H{o}s-Rademacher problem for the Erd\H{o}s-Ko-Rado Theorem.

\medskip

This question is also deeply connected to the Kneser graph.  The Kneser graph $K(n,k)$ has vertices $V = \binom{[n]}{k}$, with vertices $X$ and $Y$ adjacent if and only if the sets are disjoint.  An intersecting set system corresponds to an independent set in the Kneser graph.  Question \ref{que:ahlswede} is thus asking which $s$ vertices of the Kneser graph induce the smallest number of edges.  Since the Kneser graph is regular, this is equivalent to finding the largest bipartite subgraph of $K(n,r)$ with one part of size $s$.  Kneser graphs have been extensively studied, and the problem of determining their largest bipartite subgraphs was first raised by Poljak and Tuza in \cite{poltuz87}.

\medskip

In 2003, Bollob\'as and Leader \cite{bollead03} presented a new proof of Theorem \ref{thm:erintersect}, by relaxing the problem to a continuous version and analyzing fractional set systems.  They further considered Question \ref{que:ahlswede}, and conjectured that for small systems, the initial segment of the lexicographical ordering on $\binom{[n]}{k}$ should be optimal.  In the lexicographical ordering, we say $A < B$ if $\min( A \Delta B ) \in A$; that is, we prefer sets with smaller elements.  More generally, Bollob\'as and Leader conjectured that all the extremal systems should take the form of what they named $\ell$-balls, as defined below.  Note that a $1$-ball is an initial segment of the lexicographical ordering.

\begin{CONJ}[Bollob\'as-Leader \cite{bollead03}, 2003] \label{conj:bollead}
One of the systems $\cA_{r,l} = \{ A \in \binom{[n]}{k} : \card{A \cap [r]} \ge \ell \}$ minimizes the number of disjoint pairs.
\end{CONJ}

When $k=2$, we can think of a set system in $\binom{[n]}{2}$ as a graph on $n$ vertices, and are then asking for which graphs of a given size minimize the number of disjoint pairs of edges.  This problem was solved by Ahslwede and Katona \cite{ahlkat78} in 1978, who showed that the extremal graphs were always either the union of stars (a collection of vertices connected to all the other vertices), or their complement.  In that paper, they asked for a different generalization, namely which $k$-uniform set systems minimize the number of $(k-1)$-disjoint pairs.  This is the Erd\H{o}s-Rademacher problem for $t$-intersecting systems when $t = k-1$, and is known as the Kleitman-West problem.  It shares some connections to information theory, and while Harper solved a continuous approximation to the problem \cite{harper91}, an exact solution appears difficult to obtain.  Indeed, a natural conjecture of Kleitman for this problem has been proven to be untrue \cite{ahlcai99}.

\subsection{Our results}

Our main result verifies Conjecture \ref{conj:bollead} for small systems, showing that initial segments of the lexicographical ordering minimize the number of disjoint pairs.  We denote by $\cL_{n,k}(s)$ the first $s$ sets in the lexicographical ordering on $\binom{[n]}{k}$.  Note that the size of $\ell$ full stars, say with centers $\{1,2,\hdots, \ell\}$, is $\binom{n}{k} - \binom{n-\ell}{k}$.  The following theorem shows that provided $n$ is large enough with respect to $k$ and $\ell$, it is optimal to take sets from the first $\ell$ stars.

\begin{THM}\label{thm:disjpairs}
Provided $n > 108 k^2 \ell (k + \ell)$ and $0 \le s \le \binom{n}{k} - \binom{n - \ell}{k}$, $\cL_{n,k}(s)$ minimizes the number of disjoint pairs among all systems of $s$ sets in $\binom{[n]}{k}$.
\end{THM}

As a by-product of our proof, we shall obtain a characterization of the extremal systems in this range, which we provide in Proposition \ref{prop:structure}.  Corollary \ref{cor:largefam} shows that we can also use Theorem \ref{thm:disjpairs} to determine which systems are optimal when $s$ is very close to $\binom{n}{k}$.

We further show that $\cL_{n,k}(s)$ also minimizes the number of $q$-matchings.

\begin{THM} \label{thm:qmatchings}
Provided $n > n_1(k,q,\ell)$ and $0 \le s \le \binom{n}{k} - \binom{n - \ell}{k}$, $\cL_{n,k}(s)$ minimizes the number of $q$-matchings among all systems of $s$ sets in $\binom{[n]}{k}$.
\end{THM}

Finally, we extend our methods to determine which systems minimize the number of $t$-disjoint pairs.  When $t = k-1$, this provides a partial solution to the problem of Kleitman and West.  When $n$ is large with respect to $k,t$ and $\ell$, an extremal system is contained in the union of $\ell$ full $t$-stars.  As we shall discuss in Section \ref{sec:tdisjoint}, not all such unions are isomorphic, and once again it is the lexicographical ordering that is optimal.

\begin{THM} \label{thm:tdisjoint}
Provided $n \ge n_2(k,t,\ell)$ and $0 \le s \le \binom{n-t+1}{k-t+1} - \binom{n-t-\ell+1}{k-t+1}$, $\cL_{n,k}(s)$ minimizes the number of $t$-disjoint pairs among all systems of $s$ sets in $\binom{[n]}{k}$.
\end{THM}

We again characterize all extremal systems in Corollary \ref{cor:toptimal}.

\subsection{Outline and notation}

The remainder of the paper is organized as follows.  In Section \ref{sec:disjpairs} we study disjoint pairs, and prove Theorem \ref{thm:disjpairs}.  In Section \ref{sec:qmatchings}, we consider the number of $q$-matchings, and prove Theorem \ref{thm:qmatchings}.  In Section \ref{sec:tdisjoint}, we extend our results to $t$-disjoint pairs, proving Theorem \ref{thm:tdisjoint}.  In the final section we present some concluding remarks and open problems.

We denote by $[n]$ the set of the first $n$ positive integers, and use this as the ground set for our set systems.  Given a set $X$, $\binom{X}{k}$ is the system of all $k$-subsets of $X$.  The number of disjoint pairs between two systems $\cF$ and $\cG$ is given by $\disj(\cF, \cG) = \card{ \{ (F, G) \in \cF \times \cG: F \cap G = \emptyset \} }$, and the number of disjoint pairs within a system $\cF$ is denoted by $\disj(\cF) = \frac12 \disj(\cF, \cF)$.

For given $n,k$ and $s$, we let $\disj(n,k,s)$ denote the minimum of $\disj(\cF)$ over all $k$-uniform set systems on $[n]$ of size $s$.  We define $\qdisj(\cF)$ and $\qdisj(n,k,s)$ similarly for the number of $q$-matchings in set systems, and $\tdisj(\cF), \tdisj(\cF,\cG),$ and $\tdisj(n,k,s)$ for the number of $t$-disjoint pairs.

Given any set system $\cF$, and a set $X \subset [n]$, we let $\cF(X) = \{F \in \cF: X \subset F\}$ be those sets in the system containing $X$.  If $X$ is a singleton, we shall drop the set notation, and write $\cF(x)$.  Finally, we define a $\emph{cover}$ of a system to be a set $X$ with $\cup_{x \in X} \cF(x) = \cF$; that is, a set of elements that touches every set.  A \emph{$t$-cover} is a collection of $t$-sets such that every set in the system contains one of the $t$-sets.

\section{Disjoint Pairs} \label{sec:disjpairs}

In this section we will show that for small systems, initial segments of the lexicographical ordering, $\cL_{n,k}(s)$, minimize the number of disjoint pairs.  Note that when $s \le \binom{n-1}{k-1}$, $\cL_{n,k}(s)$ is a star, which is an intersecting system and thus clearly optimal.  The following result of Katona et al \cite{katona12} shows that if we add one set to a full star, the resulting system will also be optimal.

\begin{PROP}\label{prop:onesetmore}
Suppose $n \ge 2k$.  Any system $\cF \subset \binom{[n]}{k}$ with $|\cF| = \binom{n-1}{k-1} + 1$ contains at least $\binom{n-k-1}{k-1}$ disjoint pairs.
\end{PROP}

Our first theorem shows that as we add sets to the system, we should try to cover our system with as few stars as possible, as is achieved by $\cL_{n,k}(s)$.  Later, in Proposition \ref{prop:structure}, we shall precisely characterize all extremal systems.  We begin by recalling the statement of the theorem.

\setcounter{section}{1}
\setcounter{THM}{5}

\begin{THM}
Provided $n > 108 k^2 \ell (k + \ell)$ and $0 \le s \le \binom{n}{k} - \binom{n - \ell}{k}$, $\cL_{n,k}(s)$ minimizes the number of disjoint pairs among all systems of $s$ sets in $\binom{[n]}{k}$.
\end{THM}

\setcounter{section}{2}
\setcounter{THM}{1}

In our notation, the above theorem gives $\disj(n,k,s) = \disj(\cL_{n,k}(s))$ for such values of $s$.  Let $1 \le r \le \ell$ be such that $\binom{n}{k} - \binom{n-r+1}{k} < s \le \binom{n}{k} - \binom{n-r}{k}$.  Since $\cL_{n,k}(s)$ contains all sets meeting $\{1,2,\hdots,r-1\}$, and all the remaining sets contain $r$, we can provide an explicit formula for the minimum number of disjoint pairs.
\begin{align*}
	\disj(n,k,s) = \disj(\cL_{n,k}(s)) &= \sum_{i=1}^r \sum_{\substack{F \in \cL_{n,k}(s) \\ \min F = i}} \left| \{ G \in \cL_{n,k}(s): \min G < i, F \cap G = \emptyset \} \right| \\
	&= \sum_{i=2}^{r-1} \binom{n-i}{k-1} \sum_{j=1}^{i-1} \binom{n-j-k}{k-1} + \left( s - \sum_{i=1}^{r-1} \binom{n-i}{k-1} \right) \sum_{j=1}^{r-1} \binom{n-j-k}{k-1}.
\end{align*}

It will be useful to have a simpler upper bound on $\disj(n,k,s)$.  Note that we can assign each set in $\cL_{n,k}(s)$ to an element of $[r]$ it contains.  It can then only be disjoint from sets assigned to different elements.  In the worst case, an equal number of sets is assigned to each element, giving the bound
\begin{equation} \label{ineq:upbound}
\disj(n,k,s) \le \binom{r}{2} \left( \frac{s}{r} \right)^2 = \frac12 \left( 1 - \frac{1}{r} \right) s^2.
\end{equation}

We shall often require bounds on $\binom{n-2}{k-2}$ in terms of $s$.  Since $\cL_{n,k}(s)$ contains all sets meeting $[r-1]$ and $n$ is large, the Bonferroni Inequalities give
\begin{align} \label{ineq:lowerorder}
	s = |\cL_{n,k}(s)| &\ge (r-1) \binom{n-1}{k-1} - \binom{r-1}{2} \binom{n-2}{k-2} \notag \\
	&= \left( \frac{(r-1)(n-1)}{k-1} - \binom{r-1}{2} \right) \binom{n-2}{k-2}	\ge \frac{rn}{3k} \binom{n-2}{k-2}.
\end{align}

Our proof of Theorem \ref{thm:disjpairs} will proceed according to the following steps.  First we shall argue that if a family $\cF$ has at most $\frac{1}{2} \left(1 - \frac{1}{r} \right) s^2$ disjoint pairs, then it must contain a popular element; that is, some $x \in [n]$ contained in many sets of $\cF$.  The second step consists of a series of arguments to show that $\cF$ can be covered by $r$ elements.  The final step will then show that among all families that can be covered by $r$ elements, $\cL_{n,k}(s)$ minimizes the number of disjoint pairs.

\begin{proof}[Proof of Theorem \ref{thm:disjpairs}]
We prove the theorem by induction on $n$ and $s$.  For the base case, suppose $0 \le s \le \binom{n}{k} - \binom{n-1}{k} = \binom{n-1}{k-1}$.  In this range, $\cL_{n,k}(s)$ is a star, consisting only of sets containing $1$, and thus obviously minimizes the number of disjoint pairs.

\medskip

For the induction step, let $\cF$ be an extremal family with $|\cF| = s > \binom{n-1}{k-1}$, and so $r \ge 2$.  Suppose first that $\cF$ contains a full star.  Without loss of generality, we may assume $\cF$ has all sets containing $1$, so $\cF(1) = \left\{ F \in \binom{[n]}{k}: 1 \in F \right\}$.  Since $\cF(1)$ is an intersecting family, we have $\disj(\cF) = \disj \left( \cF(1), \cF \setminus \cF(1) \right) + \disj \left( \cF \setminus \cF(1) \right)$.  Now any set $F \in \cF$ with $1 \notin F$ is disjoint from exactly $\binom{n-k-1}{k-1}$ sets in $\cF(1)$, giving $\disj\left( \cF(1), \cF \setminus \cF(1) \right) = \left| \cF \setminus \cF(1) \right| \binom{n-k-1}{k-1} = \left( s - \binom{n-1}{k-1} \right) \binom{n-k-1}{k-1}$, regardless of the structure of $\cF \setminus \cF(1)$.  Since $\cF \setminus \cF(1)$ is a family of $s - \binom{n-1}{k-1}$ sets in $[n] \setminus \{1\}$, our induction hypothesis implies $\disj \left(\cF \setminus \cF(1) \right)$ is minimized by the initial segment of the lexicographical order.  Since $\cL_{n,k}(s)$ consists of all sets containing $1$, and the initial segment of the lexicographical order on $[n] \setminus \{1\}$, it follows that $\cL_{n,k}(s)$ is optimal, as claimed.

\medskip

Hence we may assume that $\cF$ does not contain any full star.  Consequently, given any $F \in \cF$ and $x \in [n]$, we may replace $F$ by a set containing $x$.

\medskip

\noindent \emph{Step 1:}  Show there exists some $x \in [n]$ with $|\cF(x)| \ge \frac{s}{3r}$.

\medskip

We begin by showing there cannot be too many moderately popular elements.

\begin{CLAIM} \label{clm:smallset}
$\left| \{ x : |\cF(x)| \ge \frac{s}{3kr} \} \right| < 6kr$.
\end{CLAIM}

\begin{proof}
Suppose not, and consider $X \subset \{ x : |\cF(x)| \ge \frac{s}{3kr} \}$ with $|X| = 6kr$.  Using \eqref{ineq:lowerorder}, we have
\begin{align*}
	s = | \cF | \ge \left| \cup_{x \in X} \cF(x) \right| &\ge \sum_{x \in X} | \cF(x) | - \sum_{x,y \in X} | \cF(x) \cap \cF(y) | \\
	&\ge |X| \cdot \frac{s}{3kr} - \binom{|X|}{2} \binom{n-2}{k-2} \\
	&\ge 2s - 18k^2r^2 \cdot \frac{3k}{nr} s = \left( 2 - \frac{54k^3r}{n} \right) s.
\end{align*}
Since $n > 54 k^3 r$, we reach a contradiction.
\end{proof}

We now show the existence of a popular element.

\begin{CLAIM} \label{clm:popelem}
There is some $x \in [n]$ with $|\cF(x)| > \frac{s}{3r}$.
\end{CLAIM}

\begin{proof}
Since $\cF$ is extremal, we must have $\frac12 \left(1 - \frac{1}{r}\right) s^2 \ge \disj(\cL_{n,k}(s)) \ge \disj(\cF) = \frac12 \sum_{F \in \cF} \disj(F, \cF)$.

Now $\disj(F,\cF) = s - \left| \cup_{x \in F} \cF(x) \right| \ge s - \sum_{x \in F} \left| \cF(x) \right|$, and so we have
\[ \left(1 - \frac{1}{r} \right)s^2 \ge \sum_{F \in \cF} \left( s - \sum_{x \in F} \left| \cF(x) \right| \right) = s^2 - \sum_{F \in \cF} \sum_{x \in F} \left| \cF(x) \right| = s^2 - \sum_x \left| \cF(x) \right|^2. \]

Let $X = \{ x : | \cF(x) | \ge \frac{s}{3kr} \}$, and note that by the previous claim, $|X| < 6kr$.  Moreover, without loss of generality, suppose $1$ is the most popular element, so $|\cF(x)| \le |\cF(1)|$ for all $x$.  We split the above sum into those $x \in X$ and those $x \notin X$, giving
\begin{equation} \label{ineq:popelem}
	\frac{s^2}{r} \le \sum_{x \in X} | \cF(x) |^2 + \sum_{x \notin X} | \cF(x) |^2 \le |\cF(1)| \sum_{x \in X} |\cF(x)| + \frac{s}{3kr} \sum_{x \notin X} | \cF(x) |.
\end{equation}

We bound the first sum by noting that
\[ \sum_{x \in X} | \cF(x) | \le \left| \cup_{x \in X} \cF(x) \right| + \sum_{\{x,y\} \subset X} \left| \cF(x) \cap \cF(y) \right| \le s + \binom{|X|}{2} \binom{n-2}{k-2} \le \left( 1 + \frac{54k^3r}{n} \right) s \le 2s, \]
using \eqref{ineq:lowerorder} and our bound on $n$.  The second sum is bounded by $\sum_{x \notin X} | \cF(x) | \le \sum_x | \cF(x) | = ks$.  Substituting these bounds in \eqref{ineq:popelem} gives $\frac{s^2}{r} \le 2 | \cF(1) | s + \frac{s^2}{3r}$, and so $|\cF(1)| \ge \frac{s}{3r}$, as required.
\end{proof}

This concludes Step 1.

\medskip

\noindent \emph{Step 2:}  Show there is a cover of size $r$.

\medskip

We begin by using the existence of a popular element to argue that there is a reasonably small cover, and then provide a number of claims that together imply an extremal family must in fact be covered by $r$ elements.

\begin{CLAIM} \label{clm:medcover}
$X = \{ x : | \cF(x) | \ge \frac{s}{3kr} \}$ is a cover for $\cF$.
\end{CLAIM}

\begin{proof}
Suppose for contradiction $X$ is not a cover.  Then there must be some set $F \in \cF$ with $F \cap X = \emptyset$, and so $\card{\cF(x)} < \frac{s}{3kr}$ for all $x \in F$.  Hence $\disj(F,\cF) = s - \card{ \cup_{x \in F} \cF(x) } \ge s - \sum_{x \in F} \card{\cF(x)} > s - \frac{s}{3r}$.  On the other hand, by Claim \ref{clm:popelem}, we may assume $\card{\cF(1)} \ge \frac{s}{3r}$.  Thus if $G$ is any set containing $1$, we have $\disj(G,\cF) \le s - \card{\cF(1)} = s - \frac{s}{3r}$.  Hence replacing $F$ with such a set $G$, which is possible since $\cF(1)$ is not a full star, would decrease the number of disjoint pairs in $\cF$, contradicting its optimality.

Hence $X$ must be a cover for $\cF$, as claimed.
\end{proof}

By Claim \ref{clm:smallset}, we have $|X| \le 6kr$.  Take a minimal subcover of $X$ containing $1$; without loss of generality, we may assume this subcover is $[m]$, for some $r \le m \le 6kr$.  We shall now proceed to show that an extremal family must have $m = r$, giving rise to the smallest possible cover.

Rather than working with the subfamilies $\cF(i)$, $i \in [m]$, we shall avoid double-counting by instead considering the subsystems $\cF^*(i) = \{ F \in \cF: \min F = i \}$.  Note that the systems $\cF^*(i)$ partition $\cF$.

\begin{CLAIM} \label{clm:almostequal}
For every $i,j \in [m]$, we have $| \cF^*(i) | \ge | \cF^*(j) | - \frac{3mk^2}{rn}s$.
\end{CLAIM}

\begin{proof}
First we claim that there is some $F \in \cF^*(i)$ with $F \cap [m] = \{i\}$.  Indeed, the number of sets in $\cF(i)$ intersecting another element in $[m]$ is less than $m\binom{n-2}{k-2} \le \frac{3mk}{rn} s \le \frac{18k^2}{n}s < \frac{s}{3kr}$.  However, since $[m]$ is a subcover of $X$ from Claim \ref{clm:medcover}, it follows that $|\cF(i)| \ge \frac{s}{3kr}$, and thus we must have our desired set $F \in \cF^*(i)$.

For any $j \in [m]\setminus \{i\}$, $F$ can intersect at most $k\binom{n-2}{k-2}$ sets in $\cF(j)$, since each of these sets must contain both $j$ and one element from $F$.  Summing over all $j$ and using \eqref{ineq:lowerorder} gives
\begin{align*}
	\disj(F, \cF) &\ge \sum_{j \neq i} \disj(F, \cF) \ge \sum_{j \neq i} \left[ | \cF^*(j) | - k \binom{n-2}{k-2} \right] \\
	&= s - |\cF^*(i)| - mk \binom{n-2}{k-2} \ge s - |\cF^*(i)| - \frac{3mk^2}{rn}s.
\end{align*}

On the other hand, if we were to replace $F$ with a set containing $j$, it would intersect at least those sets in $\cF^*(j)$, and so introduce at most $s - | \cF^*(j)|$ disjoint pairs.  Since $\cF$ is an extremal family, we must have $s - | \cF^*(j)| \ge s - | \cF^*(i)| - \frac{3mk^2}{rn} s$, or $|\cF^*(i)| \ge |\cF^*(j)| - \frac{3mk^2}{rn}s$, as required.
\end{proof}

\begin{CLAIM} \label{clm:smallcover}
$m \le 6r$.
\end{CLAIM}

\begin{proof}
We shall now bound $|\cF^*(i)|$ by taking $j = 1$ in Claim \ref{clm:almostequal}.  Recall that by Claim \ref{clm:popelem} we have $|\cF(1)| = |\cF^*(1)| \ge \frac{s}{3r}$, and from Claim \ref{clm:smallset} it follows that $m \le 6kr$.  Since $n > 108k^3r$, these bounds give
\[ | \cF^*(i) | \ge | \cF^*(1) | - \frac{3mk^2}{rn}s \ge \frac{s}{3r} - \frac{18k^3}{n} s \ge \frac{s}{6r}. \]
Since $s = \left| \cup_{i=1}^m \cF^*(i) \right| = \sum_{i=1}^m | \cF^*(i) | \ge m \cdot \frac{s}{6r}$, we must have $m \le 6r$.
\end{proof}

With this tighter bound on $m$, we are now able to better estimate the number of disjoint pairs in $\cF$, and in doing so show that we must actually have $m = r$ if $\cF$ is extremal.

\begin{CLAIM} \label{clm:mincover}
If $\cF$ minimizes the number of disjoint pairs, then $\cF$ can be covered by $r$ elements.
\end{CLAIM}

\begin{proof}
Since $\{ \cF^*(i) \}$ partitions $\cF$ into intersecting families, we have $\disj(\cF) = \sum_{i < j} \disj(\cF^*(i), \cF^*(j))$.  For $i < j$, note that every set $F \in \cF^*(j)$ 
can intersect at most $k \binom{n-2}{k-2}$ sets in $\cF^*(i)$, since those sets would have to contain one element from $F$ as well as $i$.  This shows that
$\disj(\cF^*(i), \cF^*(j)) \ge \left( 
|\cF^*(i)| - k \binom{n-2}{k-2} \right) | \cF^*(j)|$.  
Moreover, note that Claim \ref{clm:almostequal} implies the bound $|\cF^*(1)| \le \frac{s}{m} + \frac{3mk^2}{rn} s$, since we must have some $i \in [m]$ with $\card{\cF^*(i)} \le 
\frac{1}{m} \sum_{j=1}^m \card{\cF^*(j)} = \frac{s}{m}$, and $\card{\cF^*(i)} \ge \card{\cF^*(1)} - \frac{3mk^2}{rn} s$.  Thus

\begin{align*}
	\disj(\cF) &\ge \sum_{i < j} \left( | \cF^*(i)| - k\binom{n-2}{k-2} \right) | \cF^*(j) | \\
	&= \sum_{i < j} |\cF^*(i)||\cF^*(j)| - k \binom{n-2}{k-2} \sum_j (j-1) | \cF^*(j) | \\
	&\ge \frac12 \left( \left(\sum_i |\cF^*(i) | \right)^2 - \sum_i |\cF^*(i)|^2 \right) - mk \binom{n-2}{k-2} \sum_j | \cF^*(j)| \\
	&\ge \frac12 \left( s^2 - |\cF^*(1)| \sum_i |\cF^*(i)| \right) - mk \binom{n-2}{k-2} \sum_j | \cF^*(j)| \\
	&\ge \frac12 \left( s^2 - \left( \frac{s}{m} + \frac{3mk^2}{rn}s \right)s \right) - \frac{3mk^2}{rn} s^2 \\
	&= \frac12 \left(1 - \frac{1}{m} - \frac{9mk^2}{rn} \right) s^2.
\end{align*}

On the other hand, since $\cF$ is extremal, we have $\disj(\cF) \le \disj(\cL_{n,k}(s)) \le \frac12 \left(1 - \frac{1}{r} \right) s^2$, and so we must have $\frac{1}{r} \le \frac{1}{m} + \frac{9mk^2}{rn} \le \frac{1}{m} + \frac{54k^2}{n}$.  Since $n > 54k^2r(k+r) \ge 54k^2r(r+1)$, we have $\frac{54k^2}{n} < \frac{1}{r} - \frac{1}{r+1}$, and hence we require $m \le r$.  Thus $\cF$ can be covered by $r$ elements.
\end{proof}

This completes Step 2.

\medskip

\noindent \emph{Step 3:}  Show that $\cL_{n,k}(s)$ is optimal.

\medskip

We will now complete the induction argument by showing that $\cL_{n,k}(s)$ is indeed an extremal family.  From the preceding steps we know $\cF$ must be covered by $r$ elements, which we may assume to be $[r]$.  We shall now use a complementary argument to deduce the optimality of $\cL_{n,k}(s)$.

Let $\cA = \left\{ A \in \binom{[n]}{k} : A \cap [r] \neq \emptyset \right\}$ be all sets meeting $[r]$, so we have $\cF \subset \cA$.  Let $\cG = \cA \setminus \cF$.  We have
\[ \disj(\cF) = \disj(\cA) - \disj(\cG, \cA) + \disj(\cG), \]
since only disjoint pairs contained in $\cF$ survive on the right-hand side.

Since $\disj(\cA)$ is determined solely by $r$, and hence $s$, but is independent of the structure of $\cF$, we may treat that term as a constant.

We have $\disj(\cG,\cA) = \sum_{G \in \cG} \disj(G, \cA)$.  For any $G$, $\disj(G, \cA)$ is determined by $| G \cap [r] |$, and is maximized when $|G \cap [r]| = 1$.  For $\cF = \cL_{n,k}(s)$, we have $G \cap [r] = \{r\}$ for all $G \in \cG$, and so $\cL_{n,k}(s)$ maximizes $\disj(\cG,\cA)$.

Finally, we obviously have $\disj(\cG) \ge 0$, with equality in the case of $\cF = \cL_{n,k}(s)$.

\medskip

Hence it follows that $\cL_{n,k}(s)$ minimizes the number of disjoint pairs, as claimed.  This completes the proof.
\end{proof}

This proof also allows us to characterize all extremal systems.

\begin{PROP} \label{prop:structure}
Provided $n > 108k^2r(k+r)$ and $\binom{n}{k} - \binom{n-r+1}{k} \le s \le \binom{n}{k} - \binom{n-r}{k}$, then a set system $\cF \subset \binom{[n]}{k}$ of size $s$ minimizes the number of disjoint pairs if and only if it has one of the two following structures:
\begin{itemize}
	\item[(i)] $\cF$ contains $r-1$ full stars, with the remaining sets forming an intersecting system, or
	\item[(ii)] $\cF$ has a cover $X$ of size $r$, and if $\cG = \left\{ G \in \binom{[n]}{k} \setminus \cF : G \cap X \neq \emptyset \right\}$, then $\cG$ is intersecting, and $\card{G \cap X} = 1$ for all $G \in \cG$.
\end{itemize}
\end{PROP}

\begin{proof}
We prove the proposition by induction on $n$ and $s$.  If $0 \le s \le \binom{n-1}{k-1}$, then clearly a system is extremal if and only if it is intersecting, as there need not be any disjoint pairs.  Since $r=1$ for this value of $s$, this is covered by case (i).

For the induction step, note that if $\cF$ is extremal and contains a full star, say $\cF(1)$, then $\cF \setminus \cF(1)$ must also be extremal.  Applying the induction hypothesis gives the result, since adding a full star to either (i) or (ii) preserves the structure.

Hence we may assume there is no full star.  Claim \ref{clm:mincover} then shows that $\cF$ has a cover of size $r$, while the complementary argument from Step 3 gives the above characterization of the system $\cG$.
\end{proof}

A similar complementary argument allows us to use Theorem \ref{thm:disjpairs} to also determine the extremal systems when $s = |\cF|$ is very large.

\begin{COR} \label{cor:largefam}
Provided $n > 108 k^2 \ell (k + \ell)$ and $\binom{n-\ell}{k} \le s \le \binom{n}{k}$, $\binom{[n]}{k} \setminus \cL_{n,k} \left( \binom{n}{k} - s \right)$ minimizes the number of disjoint pairs.
\end{COR}

\begin{proof}
Let $\cA = \binom{[n]}{k}$ be the collection of all $k$-sets in $[n]$, and let $\cF \subset \cA$ be any system of $s$ sets.  Write $\cG = \cA \setminus \cF$ for those sets not in $\cF$.  As in Step 3 of the proof of Theorem \ref{thm:disjpairs}, we have $\disj(\cF) = \disj(\cA) - \disj(\cG,\cA) + \disj(\cG)$.

Since every set is disjoint from $\binom{n-k}{k}$ other sets in $\cA$, we have $\disj(\cF) = \frac12 \binom{n}{k} \binom{n-k}{k} - |\cG| \binom{n-k}{k} + \disj(\cG) = \left( s - \frac12 \binom{n}{k} \right) \binom{n-k}{k} + \disj(\cG)$.  Hence it is apparent that $\cF$ minimizes the number of disjoint pairs over all systems of size $s$ if and only if its complement, $\cG$, minimizes the number of disjoint pairs over all systems of size $\binom{n}{k} - s$.

By Theorem \ref{thm:disjpairs}, we know the initial segment of the lexicographical order is optimal when $0 \le s \le \binom{n}{k} - \binom{n-\ell}{k}$, and so it follows that the complement of the lexicographical order (which is isomorphic to the colexicographical order) is optimal when $\binom{n-\ell}{k} \le s \le \binom{n}{k}$.
\end{proof}

\section{$q$-matchings} \label{sec:qmatchings}

In this section, we determine which set systems minimize the number of $q$-matchings.  This extends Theorem \ref{thm:disjpairs}, which is the case $q=2$.  Note that when $|\cF| = s \le \binom{n}{k} - \binom{n-q+1}{k}$, the lexicographical initial segment does not contain any $q$-matchings, as all sets meet $[q-1]$.  Indeed, this is known to be the largest such family when $n > (2q-1)k-q$, as proven by Frankl \cite{frankl13}.  We shall show that, provided $n$ is suitably large, $\cL_{n,k}(s)$ continues to be optimal for families of size up to $\binom{n}{k} - \binom{n - \ell}{k}$.  Unlike for Theorem \ref{thm:disjpairs}, we have made no attempt to optimize the dependence of $n$ on the other parameters.  We provide our calculations in asymptotic notation for ease of presentation, where we fix the parameters $k, \ell$ and $q$ to be constant and let $n \rightarrow \infty$.  However, our result should certainly hold for $n > C \ell^2 k^5 ( \ell^2 + k^2) e^{3q}$.

\medskip

Our proof strategy will be very similar to before: we will first find a popular element, deduce the existence of a smallest possible cover, and then use a complementary argument to show that the initial segment of the lexicographical order is optimal.  The main difference is in the definition of \emph{popular} - rather than considering how many sets contain the element $x$, we shall be concerned with how many $(q-1)$-matchings have a set containing $x$.  To this end, we introduce some new notation.  Given a set family $\cF$, and a set $F$, let $\cF^{(q)}(F)$ denote the number of $q$-matchings $\{F_1, F_2, \hdots, F_q\}$ in $\cF$ with $\cup_{i=1}^q F_i \cap F \neq \emptyset$.  Similarly, for some $x \in [n]$, we let $\cF^{(q)}(x) = \cF^{(q)}(\{x\})$ be the number of $q$-matchings with $x \in \cup_{i=1}^q F_i$.

\setcounter{section}{1}
\setcounter{THM}{6}

\begin{THM}
Provided $n > n_1(k,q,\ell)$ and $0 \le s \le \binom{n}{k} - \binom{n - \ell}{k}$, $\cL_{n,k}(s)$ minimizes the number of $q$-matchings among all systems of $s$ sets in $\binom{[n]}{k}$.
\end{THM}

\setcounter{section}{3}
\setcounter{THM}{0}

As before, we start with some estimates on $\disj^{(q)}(\cL_{n,k}(s))$.  Let $r$ be such that $\binom{n}{k} - \binom{n-r+1}{k} < s \le \binom{n}{k} - \binom{n-r}{k}$.  We may assign each set in $\cL_{n,k}(s)$ to one of its elements in $[r]$.  Note that a $q$-matching cannot contain two sets assigned to the same element, and so to obtain a $q$-matching, we must choose sets from different elements in $[r]$.  By convexity, the worst case is when the sets are equally distributed over $[r]$, giving the upper bound
\begin{equation} \label{ineq:qupperbound}
\disj^{(q)}(\cL_{n,k}(s)) \le \binom{r}{q} \left( \frac{s}{r} \right)^q.
\end{equation}

In this case we shall also require a lower bound.  Note that $\cL_{n,k}(s)$ contains all sets meeting $[r-1]$, with the remaining sets containing $\{r\}$; suppose there are $\alpha \binom{n-1}{k-1}$ such sets.  Note that we have $s = \binom{n}{k} - \binom{n-r+1}{k} + \alpha \binom{n-1}{k-1} \le (r-1) \binom{n-1}{k-1} + \alpha \binom{n-1}{k-1}$, so $\binom{n-1}{k-1} \ge \frac{s}{r-1+\alpha}$.

 We shall consider two types of $q$-matchings - those with one of the $\alpha \binom{n-1}{k-1}$ sets that only meet $[r]$ at $r$, and those without.  For the first type, we have $\alpha \binom{n-1}{k-1}$ choices for the set containing $r$.  For the remaining sets in the $q$-matching, we will avoid any overcounting by restricting ourselves to sets that only contain one element from $[r-1]$ to avoid any overcounting.  We can then make one of $\binom{r-1}{q-1}$ choices for how the remaining $q-1$ sets will meet $[r-1]$.  For each such set, we must avoid all other elements in $[r]$ and all previously used elements, leaving us with at least $\binom{n-kq-r}{k-1}$ options.

 For the second type of $q$-matchings, there are $\binom{r-1}{q}$ ways to choose how the sets meet $[r-1]$, and then at least $\binom{n-kq-r}{k-1}$ choices for each set.

This gives $\disj^{(q)}(\cL_{n,k}(s)) \ge \alpha \binom{n-1}{k-1} \binom{r-1}{q-1} \binom{n-kq-r}{k-1}^{q-1} + \binom{r-1}{q} \binom{n-kq-r}{k-1}^q$.  Using that, for large $n$,
\[ \binom{n-kq-r}{k-1} \ge \left(1 - \frac{k(kq+r)}{n} \right) \binom{n-1}{k-1}, \]
we can simplify this expression to
\begin{align*}
	\disj^{(q)}(\cL_{n,k}(s)) &\ge \alpha \left( 1 - \frac{k(kq+r)}{n} \right)^{q-1} \binom{r-1}{q-1} \binom{n-1}{k-1}^q + \left( 1 - \frac{k(kq+r)}{n} \right)^q \binom{r-1}{q} \binom{n-1}{k-1}^q \\
	&\ge \left(1 - \frac{k(kq+r)}{n} \right)^q \left( \frac{\alpha \binom{r-1}{q-1} + \binom{r-1}{q}}{(r-1+\alpha)^q} \right) s^q \ge \left( 1 - \frac{kq(kq+r)}{n} \right) \left( \frac{\alpha \binom{r-1}{q-1} + \binom{r-1}{q}}{(r-1+\alpha)^q} \right) s^q.
\end{align*}

For fixed $s$, this fuction of $\alpha$ is monotone increasing when $0 \le \alpha \le 1$, and so the right-hand side is minimized when $\alpha = 0$.  This gives the lower bound
\begin{equation} \label{ineq:qlowerbound}
\disj^{(q)}(\cL_{n,k}(s)) \ge \left(1 - o(1) \right) \binom{r-1}{q} \left( \frac{s}{r-1} \right)^q.
\end{equation}

Having established these bounds, we now prove Theorem \ref{thm:qmatchings}.

\begin{proof}[Proof of Theorem \ref{thm:qmatchings}]
Our proof is by induction, on $n$, $q$ and $s$.  The base case for $q=2$ is given by Theorem \ref{thm:disjpairs}\footnote{Alternatively, we may use the trivial base case of $q=1$, where we merely count the number of sets.}.  As noted earlier, if $s \le \binom{n}{k} - \binom{n - q + 1}{k}$, then $\cL_{n,k}(s)$ does not contain any $q$-matchings, and hence is clearly optimal.  Hence we may proceed to the induction step, with $q \ge 3$ and $\binom{n}{k} - \binom{n - q + 1}{k} < s \le \binom{n}{k} - \binom{n - \ell}{k}$.  In particular, we have $q \le r \le \ell$ and $s = \Omega(n^{k-1})$.

\medskip

Let $\cF$ be an extremal system of size $s$.  We again first consider the case where $\cF$ contains a full star, which we shall assume to be all sets containing $1$.  We split our $q$-matchings based on whether or not they meet $1$, giving $\disj^{(q)}(\cF) = | \cF^{(q)}(1) | + \disj^{(q)}(\cF \setminus \cF(1) )$.

Note that every $(q-1)$-matching not meeting $1$ can be extended to a $q$-matching by exactly $\binom{n-k(q-1) -1 }{k-1}$ sets containing $1$, so $|\cF^{(q)}(1)| = \disj^{(q-1)}( \cF \setminus \cF(1) ) \binom{n-k(q-1)-1}{k-1}$.  By the induction hypothesis, $\disj^{(q-1)}( \cF \setminus \cF(1) )$ is minimized by the lexicographical order.  Similarly, $\disj^{(q)}( \cF \setminus \cF(1) )$ is also minimized by the lexicographical order, and hence we deduce that $\disj^{(q)}( \cF ) \ge \disj^{(q)}( \cL_{n,k}(s) )$.

Thus we may assume that $\cF$ does not contain any full stars.  Hence, for any $x \in [n]$ and any $F \in \cF$, we may replace $F$ by a set containing $x$.

\medskip

\noindent \emph{Step 1:}  Show there is a popular element $x \in [n]$, with $|\cF^{(q-1)}(x)| = \Omega(s^{q-1})$.

\medskip

A $(q-1)$-matching in $\cF$ can be extended to a $q$-matching by a set $F \in \cF$ precisely when the other $q-1$ sets do not meet $F$.  Thus $F$ is in $\disj^{(q-1)}(\cF) - |\cF^{(q-1)}(F)|$ $q$-matchings.  Summing over all $F$ gives
\[ q \cdot \disj^{(q)}(\cF) = \sum_{F \in \cF} \left( \disj^{(q-1)}(\cF) - | \cF^{(q-1)}(F) | \right) = s \cdot \disj^{(q-1)}(\cF) - \sum_{F \in \cF} | \cF^{(q-1)}(F) |. \]
By the induction hypothesis, $\disj^{(q-1)}(\cF) \ge \disj^{(q-1)}( \cL_{n,k}(s))$, and since $\cF$ is extremal, we must have $\disj^{(q)}(\cF) \le \disj^{(q)}(\cL_{n,k}(s))$.  Combining these facts with the bounds from \eqref{ineq:qupperbound} and \eqref{ineq:qlowerbound}, we get
\begin{align*}
    \sum_{F \in \cF} |\cF^{(q-1)}(F)| &= s \cdot \disj^{(q-1)}( \cF ) - q \cdot \disj^{(q)}(\cF) \ge s \cdot \disj^{(q-1)}( \cL_{n,k}(s) ) - q \cdot \disj^{(q)}( \cL_{n,k}(s) ) \\
    &\geq \left(1 - o(1) \right) \binom{r-1}{q-1} \frac{s^q}{(r-1)^{q-1}} - q \binom{r}{q} \frac{s^q}{r^q} = \Omega(s^q).
\end{align*}

Averaging over the $s$ sets in $\cF$, we must have $| \cF^{(q-1)}(F) | = \Omega( s^{q-1} )$ for some $F \in \cF$.  Since $\cF^{(q-1)}(F) = \cup_{x \in F} \cF^{(q-1)}(x)$, by averaging over the $k$ elements in $F$ we have $| \cF^{(q-1)}(x) | = \Omega(s^{q-1})$ for some $x \in F$.

\medskip

This completes Step 1.

\medskip

\noindent \emph{Step 2:} Show there is a cover of size $r$.

\medskip

From Step 1, we know there is some popular element, which we may assume to be $1$.  We start by showing the existence of a reasonably small cover.

\begin{CLAIM} \label{clm:qsmallcover}
$X = \{ x : | \cF^{(q-1)}(x) | \ge \frac{1}{k} | \cF^{(q-1)}(1) | \}$ is a cover for $\cF$.
\end{CLAIM}

\begin{proof}
Suppose for contradiction that $X$ was not a cover for $\cF$.  Then there is some set $F \in \cF$ such that $F \cap X = \emptyset$, and so $\card{\cF^{(q-1)}(x)} < \frac{1}{k} \card{\cF^{(q-1)}(1)}$ for all $x \in F$.  Recall that the number of $q$-matchings $F$ is contained in is given by
\begin{align*}
    \disj^{(q-1)}(\cF) - \card{ \cF^{(q-1)}(F)} &= \disj^{(q-1)}(\cF) - \card{ \cup_{x \in \cF} \cF^{(q-1)}(x) } \ge \disj^{(q-1)}(\cF) - \sum_{x \in F} \card{ \cF^{(q-1)}(x) } \\
        &> \disj^{(q-1)}(\cF) - \card{ \cF^{(q-1)}(1) }.
\end{align*}
On the other hand, a set containing $1$ can be in at most $\disj^{(q-1)}(\cF) - \card{\cF^{(q-1)}(1)}$ $q$-matchings.  Since $\cF(1)$ is not a full star, we may replace $F$ with a set containing $1$, and would then decrease the number of $q$-matchings in $\cF$.  This contradicts the optimality of $\cF$, and it follows that $X$ is a cover.
\end{proof}

Having shown that this set $X$ is a cover, we now show that $X$ is not too big; its size is bounded by a function of $k$, $q$ and $\ell$.

\begin{CLAIM} \label{clm:qsmallset}
$|X| = O(1)$.
\end{CLAIM}

\begin{proof}
As there can be at most $s^{q-1}$ $(q-1)$-matchings in $\cF$, we have
\[ \frac{1}{k} | \cF^{(q-1)}(1) | |X| \le \sum_{x \in X} | \cF^{(q-1)}(x)| \le \sum_{x \in [n]} | \cF^{(q-1)}(x) | = k(q-1) \disj^{(q-1)}(\cF) \le k(q-1)s^{q-1}. \]

Since $|\cF^{(q-1)}(1)| = \Omega(s^{q-1})$, this gives $|X| = O(1)$, as required.
\end{proof}

Now take a minimal subcover of $X$, which we may assume to be $[m]$, where $m = O(1)$.  We shall shift our focus from $(q-1)$-matchings to the individual sets themselves.  For each $i \in [m]$, we shall let $\cF^-(i) = \{ F \in \cF : F \cap [m] = \{i\} \}$ be those \emph{sets} in $\cF$ that meet $[m]$ precisely at $i$; by the minimality of the cover, these subsystems are non-empty.  Since any set in $\cF(i) \setminus \cF^-(i)$ must contain not just $i$ but also some other element in $[m]$, we have $\card{\cF^-(i)} \ge \card{ \cF(i) } - m \binom{n-2}{k-2} = \card{ \cF(i) } - o(s)$.

We will now show that for an extremal system, we must have $m = r$.  We first require the following claim.

\begin{CLAIM} \label{clm:qalmostequal}
For any $i,j \in [m]$, we have $|\cF(i)| = |\cF(j)| + o(s)$.
\end{CLAIM}

\begin{proof}
Recall that set $F \in \cF$ contributes $\disj^{(q-1)}(\cF) - | \cF^{(q-1)}(F) |$ $q$-matchings to $\cF$.  By estimating $|\cF^{(q-1)}(F)|$ for sets containing $i$ or $j$, we shall show that if $| \cF(i) |$ and $|\cF(j)|$ are very different, then we can decrease the number of $q$-matchings by shifting sets.

\medskip

Consider a set $F \in \cF^-(i)$.  We wish to bound $\card{\cF^{(q-1)}(F)}$.

For every $(q-1)$-matching in $\cF^{(q-1)}(F)$, we must have at least one of the sets in the $(q-1)$-matching meeting $F$.  Either this set can contain $i$, in which case there are $|\cF(i)|$ possibilities, or it contains some element in $F \setminus \{i\}$, as well as some element in $[m]$.  However, the number of options in the latter case is at most $mk\binom{n-2}{k-2} = o(s)$.  We can then count the number of possibilities for the other sets in the matching just as we did when establishing the inequalities \eqref{ineq:qupperbound} and \eqref{ineq:qlowerbound}.  First we choose representatives $A \subset [m] \setminus \{i\}$ for the other $q-2$ sets, and then we choose sets corresponding to the given elements; that is, $H \in \cF(a)$ for all $a \in A$.  This provides an overestimate for $\card{\cF^{(q-1)}(F)}$, as some of these collections of $q-1$ sets may not be disjoint, while some are counted multiple times.  However, we do obtain the upper bound
\begin{equation} \label{ineq:qmatchupper}
    |\cF^{(q-1)}(F)| \le \left( 1 + o(1) \right) |\cF(i)| \sum_{A \subset \binom{[m] \setminus \{i\}}{q-2}} \prod_{a \in A} |\cF(a)|.
\end{equation}

We now consider replacing $F$ by some set $G$ containing $j$, and determine how many new $q$-matchings would be formed.  The number of $q$-matchings $G$ contributes is $\disj^{(q-1)}(\cF) - \card{\cF^{(q-1)}(G)} \le \disj^{(q-1)}(\cF) - \card{\cF^{(q-1)}(j)}$, since $j \in G$.

To bound $\card{\cF^{(q-1)}(j)}$, note that we can form $(q-1)$-matchings containing $j$ by first choosing a set from $\cF^-(j)$, then choosing a set of $q-2$ other representatives $A \subset [m] \setminus \{j\}$, and choosing disjoint sets $H \in \cF^-(a)$, $a \in A$.  To ensure the sets we choose are disjoint, we must avoid any elements we have already used.  There can be at most $k(q-1)$ such elements, and so we have to avoid at most $\binom{n-1}{k-1} - \binom{n-k(q-1)-1}{k-1} \le k(q-1) \binom{n-2}{k-2} = o(s)$ sets each time.  By choosing the sets from $\cF^-(a)$, and not $\cF(a)$, we ensure there is no overcounting, as each such $(q-1)$-matching has a unique set of representatives in $[m]$.  Thus we have the bound
\begin{equation} \label{ineq:qmatchlower}
\card{\cF^{(q-1)}(j)} \ge \card{\cF^-(j)} \sum_{A \in \binom{[m] \setminus \{j\}}{q-2}} \prod_{a \in A} \left( \card{\cF^-(a)} - o(s) \right) = \left(1 - o(1) \right) \card{\cF(j)} \sum_{A \in \binom{[m] \setminus \{j\}}{q-2}} \prod_{a \in A} \card{\cF(a)},
\end{equation}
since $\card{\cF^-(a)} = \card{\cF(a)} - o(s)$ for all $a \in [m]$.

\medskip

Since $\cF$ is optimal, we must have $\card{\cF^{(q-1)}(F)} \ge \card{\cF^{(q-1)}(G)}$.  Comparing \eqref{ineq:qmatchupper} and \eqref{ineq:qmatchlower}, we find
\[ \left( 1 + o(1) \right) |\cF(i)| \sum_{A \subset \binom{[m] \setminus \{i\}}{q-2}} \prod_{a \in A} |\cF(a)| \ge  \left(1 - o(1) \right) \card{\cF(j)} \sum_{A \in \binom{[m] \setminus \{j\}}{q-2}} \prod_{a \in A} \card{\cF(a)}. \]

Some terms appear on both sides of the inequality, and so taking the difference gives
\[ \left( \card{\cF(i)} - \card{\cF(j)} \right) \sum_{A \subset \binom{[m] \setminus \{i,j\}}{q-2}} \prod_{a \in A} \card{\cF(a)} \ge o(s^{q-1}). \]

This implies $\card{\cF(i)} \ge \card{\cF(j)} + o(s)$.  By symmetry, the reverse inequality also holds, and thus $\card{\cF(i)} = \card{\cF(j)} + o(s)$, as required.
\end{proof}

Note that we have $s = \card{\cF} = \card{ \cup_{i \in [m]} \cF(i)} \ge \sum_{i=1}^m \card{\cF(i)} - \sum_{i < j} \card{\cF(i) \cap \cF(j)}$.  Since $\card{\cF(i) \cap \cF(j)} \le \binom{n-2}{k-2} = o(s)$ for all $i,j$, it follows that $\sum_{i =1}^m \card{\cF(i)} = s + o(s)$.  Claim \ref{clm:qalmostequal} shows that all the stars have approximately the same size, and so $\card{\cF(i)} = \frac{s}{m} + o(s)$ for each $1 \le i \le m$.  We can now show that we have a smallest possible cover.

\begin{CLAIM} \label{clm:qmincover}
If $\cF$ is extremal, then $\cF$ can be covered by $r$ elements.
\end{CLAIM}

\begin{proof}
Now that we have control over the sizes of the subsystems $\cF(i)$, we can estimate the number of $q$-matchings the system contains.  As in our calculations for Claim \ref{clm:qalmostequal}, we can obtain a $q$-matching by choosing a collection $A$ of $q$ elements in $[m]$, and then choosing sets from the corresponding subsystems $\cF(a)$, $a \in A$.  In order for this choice of sets to form a $q$-matching, each set we choose should avoid the elements of the previously chosen sets, of which there can be at most $k(q-1)$.  Moreover, to avoid overcounting, we shall choose sets from $\cF^-(a)$, and so shall avoid the other $m-1$ elements of $[m]$.  Thus, for a given $a \in A$, the forbidden sets are those containing $a$, and one of at most $k(q-1) + m-1$ other elements, and so we forbid at most $\left( k (q-1) + m-1 \right) \binom{n-2}{k-2} = o(s)$ sets.  Thus we have
\[ \disj^{(q)}(\cF) \ge \sum_{A \in \binom{[m]}{q}} \prod_{a \in A} \left( |\cF(a)| - o(s) \right) = \left(1 - o(1) \right) \binom{m}{q} \left( \frac{s}{m} \right)^q. \]

On the other hand, since $\cF$ is extremal, we must have $\disj^{(q)}(\cF) \le \disj^{(q)}(\cL_{n,k}(s)) \le \binom{r}{q} \left( \frac{s}{r} \right)^q$.  As $\binom{m}{q} \left( \frac{s}{m} \right)^q$ is increasing in $m$, these bounds imply we must have $m = r$.
\end{proof}

This concludes Step 2.

\medskip

\noindent \emph{Step 3:} Show that $\cL_{n,k}(s)$ is optimal.

\medskip

We complete the induction by showing that $\cL_{n,k}(s)$ does indeed minimize the number of $q$-matchings.  From the previous steps, we may assume that an extremal system $\cF$ is covered by $[r]$.  As before, we shall let $\cA = \left\{A \in \binom{[n]}{k} : A \cap [r] \neq \emptyset \right\}$, so $\cF \subset \cA$, and we let $\cG = \cA \setminus \cF$.  Note that for every $G \in \cG$, $\disj^{(q-1)}(\cA) - |\cA^{(q-1)}(G)|$ counts the number of $q$-matchings in $\cA$ containing $G$.  Hence
\[ \disj^{(q)}(\cF) \ge \disj^{(q)}(\cA) - \sum_{G \in \cG} \left( \disj^{(q-1)}(\cA) - |\cA^{(q-1)}(G)| \right) = \disj^{(q)}(\cA) - |\cG| \disj^{(q-1)}(\cA) + \sum_{G \in \cG} |\cA^{(q-1)}(G)|. \]

Now the first two terms are independent of the structure of $\cF$.  We claim that $\card{\cA^{(q-1)}(G)}$ is minimized when $\card{G \cap [r]} = 1$.  Indeed, fix some $G \in \cG$.  Note that the number of $(q-1)$-matchings in $\cA$ that only meet $G$ outside $[r]$ is at most $kr \binom{n-2}{k-2} s^{q-2} = o(s^{q-1})$, since we must choose one of $k$ elements of $G$ and one of $r$ elements of $[r]$ for the set to contain, and then there are at most $s^{q-2}$ choices for the remaining $q-2$ sets.  Hence almost all the $(q-1)$-matchings in $\cA^{(q-1)}(G)$ meet $G$ in $G \cap [r]$, and thus $\card{\cA^{(q-1)}(G)}$ is obviously minimized when $\card{G \cap [r]} = 1$.

\medskip

When $\cF = \cL_{n,k}(s)$, we have $G \cap [r] = \{r\}$ for all $G \in \cG$, and so the right-hand side is minimized.  Moreover, because $\cG$ is an intersecting system, it follows that every $(q-1)$-matching in $\cA$ can contain at most $1$ set from $\cG$, and so the above inequality is in fact an equality.  This shows that $\cL_{n,k}(s)$ minimizes the number of $q$-matchings.

This completes the induction step, and thus the proof of Theorem \ref{thm:qmatchings}.
\end{proof}

\section{$t$-disjoint pairs} \label{sec:tdisjoint}

We now seek a different extension of Theorem \ref{thm:disjpairs}.  Recall that we call a pair of sets $F_1, F_2$ $t$-intersecting if $\card{F_1 \cap F_2} \ge t$, and $t$-disjoint otherwise.  As shown by Wilson \cite{wilson84}, provided $n \ge (k-t+1)(t+1)$, the largest $t$-intersecting system consists of $\binom{n-t}{k-t}$ sets that share a common $t$-set $X \in \binom{[n]}{t}$; we call such a system a \emph{(full) $t$-star with center $X$}.  Note that $\cL_{n,k}(\binom{n-t}{k-t})$ is itself a $t$-star with center $[t]$.  In the following theorem, we show that when $n$ is sufficiently large, the minimum number of $t$-disjoint pairs is attained by taking full $t$-stars.  In this setting, not all unions of $t$-stars are isomorphic, as the structure depends on how the centers intersect.  We show that it is optimal to have the centers be the first few sets in the lexicographical ordering on $\binom{[n]}{t}$, which is the case for $\cL_{n,k}(s)$.

\setcounter{section}{1}
\setcounter{THM}{7}

\begin{THM}
Provided $n \ge n_2(k,t,\ell)$ and $0 \le s \le \binom{n-t+1}{k-t+1} - \binom{n-t-\ell+1}{k-t+1}$, $\cL_{n,k}(s)$ minimizes the number of $t$-disjoint pairs among all systems of $s$ sets in $\binom{[n]}{k}$.
\end{THM}

\setcounter{section}{4}
\setcounter{THM}{0}

It shall sometimes be helpful to count the number of $t$-intersecting pairs instead of $t$-disjoint pairs.  Thus we introduce the notation $\tint(\cF)$ to represent the number of $t$-intersecting pairs of sets in $\cF$, and $\tint(\cF, \cG) = \card{ \{ (F,G) \in \cF \times \cG: \card{F \cap G} \ge t \} }$ to count the number of cross-$t$-intersections between $\cF$ and $\cG$.  Note that a set $F$ is $t$-intersecting with itself, since $\card{F \cap F} = k > t$.  Since $\sum_{F \in \cF} \tint(F, \cF)$ counts the $t$-intersecting pairs between distinct sets twice, and those with the same set only once, we obtain the identity $\sum_{F \in \cF} \tint(F, \cF) = 2 \tint( \cF ) - \card{\cF}$.

\medskip

We begin with a heuristic calculation that suggests why it is optimal to have full $t$-stars.  Let $\cF$ be a full $t$-star, say with center $X \in \binom{[n]}{t}$, and let $F$ be a set not containing $X$.  For a set $G$ in $\cF$ to be $t$-intersecting with $F$, $G$ must contain the $t$ elements of $X$, as well as some $t - \card{F \cap X}$ elements from $F$.  The number of such sets $G$ is maximized when $\card{F \cap X} = t-1$, giving
\begin{equation} \label{eqn:heuristic}
	\tint(F, \cF) \le (k-t+1) \binom{n-t-1}{k-t-1} = O(n^{k-t-1}) = o(\card{\cF}).
\end{equation}

Hence if a $t$-star does not contain a set $F$, $F$ is $t$-disjoint from almost all its members.  It should thus be optimal to take full $t$-stars, as that is where the $t$-intersections come from.  Indeed, this turns out to be the case.  As we shall see, for a set system $\cF$, the leading term in $\tdisj(\cF)$ is determined by the number of $t$-stars in $\cF$.  While unions of $t$-stars may be non-isomorphic, the differences only affect the lower order terms of $\tdisj(\cF)$.

\medskip

In order to prove Theorem \ref{thm:tdisjoint}, we shall require a few preliminary results.  Proposition \ref{prop:tstars} can be thought of as a rough characterization of extremal systems, as it shows that the extremal systems should be supported on the right number of $t$-stars.  To this end, it will be useful to define an \emph{almost full $t$-star} to be a $t$-star in $\cF$ containing $(1 - o(1)) \binom{n-t}{k-t}$ sets.  Formally, this means that for all fixed $k, t$ and $\ell$, there is some $\varepsilon = \varepsilon(k,t,\ell) > 0$ such that a $t$-star will be almost full if it contains $(1 - \varepsilon)\binom{n-t}{k-t}$ sets.

\begin{PROP} \label{prop:tstars}
Suppose $n \ge n_2(k,\ell,t)$, and $\binom{n-t+1}{k-t+1} - \binom{n-t-r+2}{k-t+1} < s \le \binom{n-t+1}{k-t+1} - \binom{n-t-r+1}{k-t+1}$.  If $\cF \subset \binom{[n]}{k}$ has the minimum number of $t$-disjoint pairs over all systems of $s$ sets, then either:
\begin{itemize}
    \item[(i)] $\cF$ contains $r-1$ full $t$-stars,
    \item[(ii)] $\cF$ consists of $r$ almost full $t$-stars, or
    \item[(iii)] $\cF$ consists of $r-1$ almost full $t$-stars.
\end{itemize}
\end{PROP}

Once we have determined the large-scale structure of the extremal systems, the following lemmas allow us to analyze the lower-order terms and determine that the lexicographical ordering is indeed optimal.

\medskip

Lemma \ref{lem:lexsmall} shows that of all unions of $r$ full $t$-stars, the lexicographical ordering contains the fewest sets.  This may seem to contradict the lexicographical ordering being optimal, given that the heuristic given by \eqref{eqn:heuristic} suggests that it is optimal to take as few $t$-stars as possible, and hence we might try to make the union of these stars accommodate as many sets as possible.  However, it is because there is more overlap between the lexicographical $t$-stars that there are fewer $t$-disjoint pairs between stars.

\begin{LEMMA} \label{lem:lexsmall}
Suppose $n \ge n_2(k,t,r)$, and let $\cF$ be the union of $r$ full $t$-stars in $\binom{[n]}{k}$.  Then $|\cF| \ge \binom{n-t+1}{k-t+1} - \binom{n-t-r+1}{k-t+1}$, with equality if and only if $\cF$ is isomorphic to $\cL_{n,k}\left( \binom{n-t+1}{k-t+1} - \binom{n-t-r+1}{k-t+1} \right)$.
\end{LEMMA}

The next lemma shows that if we have $r$ full $t$-stars, and add a new set to the system, we minimize the number of new $t$-disjoint pairs created when we have the lexicographical initial segment.

\begin{LEMMA} \label{lem:addset}
Suppose $n \ge n_2(k,t,r)$, let $\cL = \cL_{n,k}\left( \binom{n-t+1}{k-t+1} - \binom{n-t-r+1}{k-t+1} \right)$ be the first $r$ full $t$-stars in the lexicographical order, and let $L$ be a set containing $\{1, 2, \hdots, t-1\}$ that is not in $\cL$.  Let $\cF$ be the union of $r$ full $t$-stars 
with centers $\{X_1, X_2, \hdots, X_r\}$, and let $F$ be any $k$-set not in $\cF$.  Then $\disj_t(F,\cF) \ge \disj_t(L,\cL)$, which equality if and only if $\cF \cup \{F\}$ is isomorphic to $\cL \cup \{L\}$.
\end{LEMMA}

However, the comparison in Lemma \ref{lem:addset} is not entirely fair, as Lemma \ref{lem:lexsmall} shows that $\cL$ will have fewer sets than $\cF$, while we ought to be comparing systems of the same size.  We do this in our final lemma, in the cleanest case when the system $\cF$ is a union of full $t$-stars.

\begin{LEMMA} \label{lem:fullstars}
Suppose $n \ge n_2(k,t,r)$, let $\cF$ be the union of $r$ full $t$-stars with centers $X_i$, $1 \le i \le r$, and let $\cL = \cL_{n,k}(\card{\cF})$.  Then $\disj_t(\cF) \ge \disj_t(\cL)$, with equality if and only if $\cF$ is isomorphic to $\cL$.
\end{LEMMA}

Armed with Proposition \ref{prop:tstars} and these three lemmas, whose proofs we defer until later in this section, we now show how to deduce Theorem \ref{thm:tdisjoint}.

\begin{proof}[Proof of Theorem \ref{thm:tdisjoint}]
Let $r$ be such that $\binom{n-t+1}{k-t+1} - \binom{n-t-r+2}{k-t+1} < s \le \binom{n-t+1}{k-t+1} - \binom{n-t-r+1}{k-t+1}$.  In this range $\cL_{n,k}(s)$ consists of $r-1$ full $t$-stars, with the remaining sets forming a partial $r$th $t$-star.  If $r = 1$, then $\cL_{n,k}(s)$ is $t$-intersecting, and therefore clearly optimal.  Hence we may assume $r \ge 2$, and in particular this implies $s = \Omega(n^{k-t})$.

\medskip

Suppose $\cF$ is an optimal system of size $s$.  By analyzing the three cases in Proposition \ref{prop:tstars} in turn, we shall show that $\tdisj(\cF) \ge \tdisj(\cL_{n,k}(s))$, thus completing the proof of Theorem \ref{thm:tdisjoint}.

\medskip

\noindent \underline{Case $(i)$:}  Suppose $\cF$ contains $r-1$ full $t$-stars, whose union we shall denote by $\cF_1$, and $s_2 = s - \card{\cF_1}$ other sets, denoted by $\cF_2$.  We then have
 \begin{align*}
    \tdisj(\cF) &= \tdisj(\cF_1) + \tdisj(\cF_1,\cF_2) + \tdisj(\cF_2) \\
        &\ge \tdisj(\cF_1) + \tdisj(\cF_1,\cF_2) \\
        &= \tdisj(\cF_1) + \sum_{F \in \cF_2} \tdisj(F,\cF_1) \\
        &\ge \tdisj(\cF_1) + s_2 \cdot \tdisj(F_0,\cF_1),
 \end{align*}
where $F_0 \in \cF_2$ minimizes $\tdisj(F,\cF_1)$.

\medskip

Let $\cL = \cL_{n,k}(s)$ be the corresponding lexicographical initial segment, $\cL_1 = \cL_{n,k}(\card{\cF_1})$ be the first $\card{\cF_1}$ sets in the lexicographical ordering, and let $\cL_2 = \cL \setminus \cL_1$ be the next $s_2$ sets.  By Lemma \ref{lem:lexsmall}, it follows that $\cL_1$ consists of at least $r-1$ full $t$-stars, and so $\cL_2$ lies entirely within the $r$th lexicographical $t$-star, and is thus $t$-intersecting.  Hence
 \begin{align*}
    \tdisj(\cL) &= \tdisj(\cL_1) + \tdisj(\cL_1,\cL_2) + \tdisj(\cL_2) \\
        &= \tdisj(\cL_1) + \tdisj(\cL_1, \cL_2) \\
        &= \tdisj(\cL_1) + \sum_{L \in \cL_2} \tdisj(L,\cL_1) \\
        &\le \tdisj(\cL_1) + s_2 \cdot \tdisj(L_0,\cL_1),
 \end{align*}
where $L_0 \in \cL_2$ maximizes $\tdisj(L,\cL_1)$ (in fact, by symmetry, this is equal for all $L \in \cL_2$).

Note that $L_0$ will belong to the $r$th $t$-star of $\cL$, and hence $\tdisj(L_0, \cL_1)$ will only count $t$-disjoint pairs between $L_0$ and the union of the first $r-1$ $t$-stars of $\cL_1$.  By Lemma \ref{lem:addset}, we have $\tdisj(F_0, \cF_1) \ge \tdisj(L_0,\cL_1)$, and by Lemma \ref{lem:fullstars}, we have $\tdisj(\cF_1) \ge \tdisj(\cL_1)$, from which we deduce $\tdisj(\cF) \ge \tdisj(\cL)$, as required.

\medskip

\noindent \underline{Case $(ii)$:}  In this case we have $r$ almost full $t$-stars.  Using a complementary argument, we shall reduce this to case $(i)$.

\medskip

Suppose $\cF$ is the union of $r$ almost full $t$-stars with centers $\{ X_1, X_2, \hdots, X_r \}$, let $\cA = \cup_{i=1}^r \{ A \in \binom{[n]}{k} : X_i \subset A \}$ be the system of all sets containing some $X_i$, and let $\cG = \cA \setminus \cF$.  On account of the $t$-stars being almost full, we have $\card{\cG} = o(n^{k-t})$.

Running the same complementary argument as in the proof of Theorem \ref{thm:disjpairs}, we have
\begin{equation} \label{eqn:tcomplement}
	\tdisj(\cF) = \tdisj(\cA) - \tdisj(\cG,\cA) + \tdisj(\cG) = \tdisj(\cA) - \sum_{G \in \cG} \tdisj(G,\cA) + \tdisj(\cG).
\end{equation}

To minimize $\tdisj(\cF)$, we seek to maximize $\sum_{G \in \cG} \tdisj(G, \cA)$ while minimizing $\tdisj(\cG)$.  We shall obtain these extrema by shifting the system so that the missing sets, $\cG$, will all belong to one of the $t$-stars $\cA(X_i)$.  In this case, the shifted system, $\cF'$, will contain $r-1$ full $t$-stars.  Hence we will have reduced the problem to case $(i)$, and so $\tdisj(\cF) \ge \tdisj(\cF') \ge \tdisj(\cL_{n,k}(s))$, as desired.

\medskip

Note that when $\cG$ is a subset of one of the $t$-stars, $\cG$ is $t$-intersecting, and so $\tdisj(\cG) = 0$ is minimized.  We now show how to choose which $t$-star $\cG$ should belong to in order to maximize $\sum_{G \in \cG} \tdisj(G,\cA)$.

Since $\cA$ is of fixed size, maximizing $\tdisj(G,\cA)$ is equivalent to minimizing $\tint(G,\cA)$.  For $G \in \cG$, $\tint(G, \cA)$ is determined by the intersections $\{G \cap X_i: 1 \le i \le r\}$.  There are only a bounded number of possibilities for these intersections, and so we may choose one which minimizes $\tint(G, \cA)$, under the restriction that $X_i \subset G$ for some $i$, since $G \in \cA$.  By \eqref{eqn:heuristic}, the number of $t$-intersecting pairs between $G$ and a $t$-star it is not in is $o(s)$, and so this minimum occurs when $G$ contains some $X_i$ and no other elements from $\cup_j X_j \setminus X_i$.  The number of choices for the set $G$ is then at least $\binom{n-rt}{k-t}$, since after choosing the $t$ elements of $X_i$, we wish to avoid the remaining elements in $\cup_j X_j$, of which there are at most $(r-1)t$.  Since $\binom{n-rt}{k-t} \ge \card{\cG} = o(n^{k-t})$, we may choose all $G \in \cG$ to come from the $t$-star with center $X_i$ in order to minimize the right hand side of \eqref{eqn:tcomplement}.  We have thus resolved case $(ii)$.

\medskip

\noindent \underline{Case $(iii)$:}  In this case we have $r-1$ almost full $t$-stars.  Since the size of this system is at most $(r-1) \binom{n-t}{k-t}$, while the size of the first $r-1$ $t$-stars in $\cL_{n,k}(s)$ is $\binom{n-t+1}{k-t+1} - \binom{n-t-r+2}{k-t+1} = (r-1) \binom{n-t}{k-t} + o(n^{k-t})$, we can conclude that $r$th partial $t$-star in $\cL_{n,k}(s)$ has only $o(n^{k-t})$ sets.

\medskip

Given the system $\cF$, we shall construct a larger system $\cF'$ by filling the $r-1$ almost full $t$-stars.  Suppose we have to add $s_1$ sets in order to do so.  Note that since the $t$-stars were almost full, we have $s_1 = o(n^{k-t})$.  Since each of the $s_1$ sets is added to an almost full $t$-star, it contributes at least $(1 - o(1)) \binom{n-t}{k-t}$ $t$-intersecting pairs.  Hence $\tint(\cF') \ge \tint(\cF) + (1 - o(1)) s_1 \binom{n-t}{k-t}$.

On the other hand, consider adding the same number of sets to the lexicographical initial segment.  The sets in $\cL_{n,k}(s + s_1) \setminus \cL_{n,k}(s)$ all belong only to the $r$th $t$-star, which has only $o(n^{k-t})$ sets.  Our calculation in \eqref{eqn:heuristic} shows that each such set also only gains $o(n^{k-t})$ $t$-intersections from the other stars, and so we have $\tint(\cL_{n,k}(s+s_1)) \le \tint(\cL_{n,k}(s)) + s_1 \cdot o(n^{k-t})$.

\medskip

Now $\cF'$ consists of $r-1$ full $t$-stars, and so by Lemma \ref{lem:fullstars}, we have $\tdisj(\cF') \ge \tdisj(\cL_{n,k}(s+s_1))$, or, equivalently, $\tint(\cF') \le \tint(\cL_{n,k}(s + s_1))$.  Thus $\tint(\cF) + (1 - o(1)) s_1 \binom{n-t}{k-t} \le \tint(\cL_{n,k}(s)) + s_1 \cdot o(n^{k-t})$, and so $\tint(\cF) \le \tint(\cL_{n,k}(s))$, with a strict inequality unless $s_1 = 0$.  This implies $\tdisj(\cF) \ge \tdisj(\cL_{n,k}(s))$, as required.

\medskip

Hence we may conclude that for any system $\cF$ with $s$ sets, we have $\tdisj(\cF) \ge \tdisj(\cL_{n,k}(s))$, proving Theorem \ref{thm:tdisjoint}.
\end{proof}

By analyzing the cases when we have equality, and using the fact that in Lemmas \ref{lem:lexsmall}, \ref{lem:addset} and \ref{lem:fullstars} we only have equality when the systems are isomorphic to the lexicographical ordering, we can characterize all extremal systems.

\begin{COR} \label{cor:toptimal}
Suppose $n \ge n_2(k, \ell, t)$, and $0 \le s \le \binom{n-t+1}{k-t+1} - \binom{n-t-\ell+1}{k-t+1}$, and $\cF \subset \binom{[n]}{k}$ minimizes the number of $t$-disjoint pairs over all systems of $s$ sets.  Then all sets $F \in \cF$ share some common $(t-1)$-set $X$, and $\cF' = \{ F \setminus X : F \in \cF \}$ minimizes the number of disjoint pairs over all systems of $s$ sets in $\binom{[n] \setminus X}{k - t + 1}$.
\end{COR}

It remains to prove the proposition and lemmas.  We begin with a proof of Proposition \ref{prop:tstars}.  The strategy will be very similar to that of Theorem \ref{thm:disjpairs}; assuming the extremal system $\cF$ does not have $r-1$ full $t$-stars, we shall show there is some popular element (that is, an element contained in many sets of $\cF$).  From this we will deduce the existence of a small cover, and shall show that either case $(ii)$ or case $(iii)$ must hold.

\begin{proof}[Proof of Proposition \ref{prop:tstars}]
We may assume that $r \ge 2$, since if $r=1$, then case $(i)$ is trivially satisfied.  We first estimate the number of $t$-intersecting pairs in $\cL_{n,k}(s)$, so that we have a lower bound on $\tint(\cF)$ for any extremal system $\cF$.

\medskip

Note that $\cL_{n,k}(s)$ consists of $r-1$ full $t$-stars, with the remaining sets forming a partial $t$-star; suppose there are $\alpha \binom{n-t}{k-t}$ such sets.  Since there are $\binom{n-t-1}{k-t-1} = o(n^{k-t})$ sets common to any two $t$-stars, it follows that $s = (r - 1 + \alpha) \binom{n-t}{k-t} + o(n^{k-t})$.

Now any two sets in the same $t$-star are $t$-intersecting, while \eqref{eqn:heuristic} shows that a set is $t$-intersecting with $o(n^{k-t})$ sets from the other $t$-stars.  Hence for any extremal system $\cF$ we have the bound
\[ \tint(\cF) \ge \tint(\cL(s)) = (r-1) \binom{ \binom{n-t}{k-t}}{2} + \binom{\alpha \binom{n-t}{k-t}}{2} + o(n^{2(k-t)}) = \frac{ r - 1 + \alpha^2 }{2} \binom{n-t}{k-t}^2 + o(n^{2(k-t)}). \]

Suppose $\cF$ contains $p$ full $t$-stars.  If $p = r-1$, then case $(i)$ holds, and we are done.  Hence we may assume $0 \le p \le r-2$.  Let $\cF_1$ be the union of the $p$ full $t$-stars, and let $\cF_2 = \cF \setminus \cF_1$ be the remaining sets.

By the same reasoning as above, we must have $\card{\cF_1} = p \binom{n-t}{k-t} + o(n^{k-t})$, and $\tint(\cF_1) = \frac12 p \binom{n-t}{k-t}^2 + o(n^{2(k-t)})$.  No set $F \in \cF_2$ is in any of the $t$-stars of $\cF_1$, and so \eqref{eqn:heuristic} gives $\tint(\cF_1,\cF_2) = \card{\cF_2} \cdot o(n^{k-t}) = o(n^{2(k-t)})$.  Thus $\tint(\cF) = \tint(\cF_1) + \tint(\cF_1,\cF_2) + \tint(\cF_2) = \frac12 p \binom{n-t}{k-t}^2 + \tint(\cF_2) + o(n^{2(k-t)})$, and hence we must have
\[ \tint(\cF_2) \ge \frac{r - p - 1 + \alpha^2}{2} \binom{n-t}{k-t}^2 + o(n^{2(k-t)}) = \Omega(n^{2(k-t)}). \]

We shall now deduce the existence of a $t$-cover of size $r-p-1$ or $r-p$ for $\cF_2$, and then show that we must fall into case $(ii)$ or $(iii)$.  The first step is to find a $t$-set that is in many members of $\cF_2$.  Note that none of the $t$-stars in $\cF_2$ are full, and hence we may shift sets in $\cF_2$.

\begin{CLAIM} \label{clm:tpopset}
There is some set $X_1 \in \binom{[n]}{t}$ with $\card{\cF_2(X_1)} = \Omega(n^{k-t})$.
\end{CLAIM}

\begin{proof}
Let $X_1 \in \binom{[n]}{t}$ be the set maximizing $\card{\cF_2(X)}$.  We have
\begin{align*}
	\tint(\cF_2) - \frac12 \card{\cF} &= \frac12 \sum_{F \in \cF_2} \tint(F,\cF_2) = \frac12 \sum_{F \in \cF_2} \card{ \cup_{X \in \binom{F}{t}} \cF_2(X) } \le \frac12 \sum_{F \in \cF_2} \sum_{X \in \binom{F}{t}} \card{ \cF_2(X) } \\
	&\le \frac12 \sum_{F \in \cF_2} \binom{k}{t} \card{ \cF_2(X_1)} = \frac12 \binom{k}{t} \card{\cF_2} \card{\cF_2(X_1)}.
\end{align*}

Since $\card{\cF_2} = (r - p - 1 + \alpha) \binom{n-t}{k-t} = O(n^{k-t})$, and $\tint(\cF_2) = \Omega(n^{2(k-t)})$, it follows that $\card{\cF_2(X_1)} = \Omega(n^{k-t})$, as desired.
\end{proof}

This allows us to find a small $t$-cover.

\begin{CLAIM} \label{clm:tcover}
$\cX = \left\{ X \in \binom{[n]}{t} : \card{\cF_2(X)} \ge \frac{1}{2\binom{k}{t}} \card{\cF_2(X_1)} \right\}$ is a $t$-cover for $\cF_2$.
\end{CLAIM}

\begin{proof}
Suppose not.  Then there is some $F \in \cF$ such that for all $X \in \binom{F}{t}$, $\card{\cF_2(X)} < \frac{1}{2\binom{k}{t}} \card{\cF_2(X_1)}$.  Thus $\tint(F, \cF_2) \le \sum_{X \in \binom{F}{t}} \card{\cF_2(X)} < \frac12 \card{ \cF_2(X_1)}$.  Since $F$ has $o(n^{k-t})$ $t$-intersecting pairs in $\cF_1$, it follows that $\tint(F,\cF) \le \frac12 \card{ \cF_2(X_1) } + o(n^{k-t})$.

If we were to replace $F$ with some set $G$ containing $X_1$, which is possible as $\cF(X_1)$ is not a full $t$-star, then we would create at least $\card{\cF_2(X_1)}$ $t$-intersecting pairs.  Since $\card{\cF_2(X_1)} = \Omega(n^{k-t})$, it follows that $\tint(G,\cF) > \tint(F,\cF)$, which contradicts $\cF$ being optimal.

Hence $\cX$ must be a $t$-cover for $\cF_2$, as claimed.
\end{proof}

\begin{CLAIM} \label{clm:tcoversmall}
$\card{\cX} = O(1)$.
\end{CLAIM}

\begin{proof}
We have
\[ \binom{k}{t} \card{\cF_2} = \sum_{F \in \cF_2} \card{ \binom{F}{t} } = \sum_{X \in \binom{[n]}{t}} \card{ \cF_2(X) } \ge \sum_{X \in \cX} \card{ \cF_2(X) } \ge \frac{1}{2 \binom{k}{t} } \card{\cF_2(X_1)} \card{\cX}. \]

Since $\card{\cF_2} = O(n^{k-t})$ and $\card{\cF_2(X_1)} = \Omega(n^{k-t})$, it follows that $\card{\cX} = O(1)$, as claimed.
\end{proof}

Hence we can write $\cX = \{X_1, X_2, \hdots, X_m\}$, where $m = O(1)$.  Note that there are at most $\binom{n-t-1}{k-t-1} = o(n^{k-t})$ sets in common between any two stars, 
while the number of sets each $t$-star contains is at 
least $\frac{1}{2 \binom{k}{t}} \card{\cF_2(X_1)}= \Omega(n^{k-t})$.  Thus in what follows, we consider only those sets in exactly one $t$-star $\cF_2(X_i)$, and shall only 
lose $o(n^{2(k-t)})$ $t$-intersecting pairs.

\begin{CLAIM} \label{clm:tequalstars}
For all $1 \le i < j \le m$, $\card{\cF_2(X_i)} = \card{\cF_2(X_j)} + o(n^{k-t})$.
\end{CLAIM}

\begin{proof}
Consider a set $F \in \cF_2(X_i)$.  $F$ is $t$-intersecting with all sets in $\cF_2(X_i)$, and, by \eqref{eqn:heuristic}, $t$-disjoint from almost all other sets.  Thus $\tint(F,\cF_2) = \card{\cF_2(X_i)} + o(n^{k-t})$.  If we were instead to replace $F$ with a set $G$ containing $X_j$, which is possible as $\cF_2(X_j)$ is not a full $t$-star, then we would create at least $\card{\cF_2(X_j)}$ new $t$-intersecting pairs.  Since $\cF$ is optimal, we must have $\card{\cF_2(X_i)} + o(n^{k-t}) \ge \card{\cF_2(X_j)}$.

By symmetry, it follows that $\card{\cF_2(X_i)} = \card{\cF_2(X_j)} + o(n^{k-t})$.
\end{proof}

Recall that we had $\card{\cF_2} = (r - p - 1 + \alpha) \binom{n-t}{k-t} + o(n^{k-t})$.  By Claim \ref{clm:tequalstars}, it follows that these sets are almost equally distributed between the $m$ $t$-stars in the $t$-cover $\cX$, and so $\card{\cF_2(X_i)} = \frac{r - p - 1 + \alpha}{m} \binom{n-t}{k-t} + o(n^{k-t})$ for each $1 \le i \le m$.  Moreover, since $m$ $t$-stars can have at most $m \binom{n-t}{k-t}$ sets, we must have $m \ge r-p-1$ if $\alpha = o(1)$, or $m \ge r - p$ if $\alpha = \Omega(1)$.

\medskip

We can now estimate $\tint(\cF_2)$.  We know every set belonging only to the $t$-star $\cF_2(X_i)$ contributes $\card{\cF_2(X_i)} + o(n^{k-t})$ $t$-intersecting pairs, while there are only $o(n^{2(k-t)})$ $t$-intersecting pairs from sets in multiple $t$-stars.  Thus
\begin{align*}
	\tint(\cF_2) &= \frac12 \sum_{F \in \cF_2} \tint(F, \cF_2) + \frac12 \card{\cF} = \frac12 \sum_{i = 1}^m \sum_{F \in \cF_2(X_i)} \tint(F,\cF_2) + o(n^{2(k-t)}) \\
	&= \frac12 \sum_{i=1}^m \card{\cF_2(X_i)} \left( \card{\cF_2(X_i)} + o(n^{k-t}) \right) + o(n^{2(k-t)}) \\
	&= \frac12 \sum_{i=1}^m \card{\cF_2(X_i)}^2 + o(n^{2(k-t)}) \\
	&= \frac{ (r - p - 1 + \alpha)^2 }{2m} \binom{n-t}{k-t}^2 + o(n^{2(k-t)})
\end{align*}

On the other hand, we had the bound
\[ \tint(\cF_2) \ge \frac{ r - p - 1 + \alpha^2}{2} \binom{n-t}{k-t}^2 + o(n^{2(k-t)}). \]

Comparing the two, we must have
\begin{equation} \label{ineq:tcompare}
	\frac{(r-p-1+\alpha)^2}{2m} \ge \frac{r-p-1+\alpha^2}{2} + o(1).
\end{equation}

Note that we can write $\frac{r-p-1+\alpha^2}{2} = \frac12 \sum_{i=1}^m x_i^2$, where
\[ x_i = \left\{ \begin{array}{ll}
	1 & 1 \le i \le r - p - 1 \\
	\alpha & i = r - p \\
	0 & r - p + 1 \le i \le m
\end{array} \right. . \]

Let $\overline{x} = \frac{1}{m} \sum_{i=1}^m x_i = \frac{r-p-1 + \alpha}{m}$.  With this definition, we then have $\frac{(r-p -1+\alpha)^2}{2m} = \frac12 m \overline{x}^2$.  Since
\[ \sum_{i=1}^m x_i^2 = m\overline{x}^2 + \sum_{i=1}^m \left( x_i - \overline{x} \right)^2, \]
for \eqref{ineq:tcompare} to hold, we must have $\sum_{i=1}^m (x_i - \overline{x})^2 = o(1)$, and thus $x_i = \overline{x} + o(1)$ for all $1 \le i \le m$.

Since $x_1 = 1$, $x_{r-p} = \alpha$, and $x_{r-p+1} = 0$, we must have $m \le r-p$.  Recalling our earlier bound $m \ge r - p - 1$, there are only two possibilities.  We could have $m = r-p$ and $\alpha = 1 - o(1)$.  In this case, each of the $r-p$ $t$-stars in $\cF_2$ has size $\frac{r-1-p+\alpha}{m} \binom{n-t}{k-t} + o(n^{k-t}) = (1 - o(1)) \binom{n-t}{k-t}$.  Combined with the $p$ full $t$-stars in $\cF_1$, we see that $\cF$ consists of $r$ almost full $t$-stars, and so we are in case $(ii)$.

The other possible solution is to have $m = r-p-1$, with $\alpha = o(1)$.  This implies $\cF_2$ consists of $r-1-p$ almost full $t$-stars, which, combined with the $p$ full $t$-stars of $\cF_1$, means $\cF$ falls under case $(iii)$.  This completes the proof of Proposition \ref{prop:tstars}.
\end{proof}

We complete this section by proving the three lemmas.  First we show that unions of lexicographical stars contain the fewest sets.

\begin{proof}[Proof of Lemma \ref{lem:lexsmall}]
Note that the first $r$ $t$-stars in the lexicographical ordering have centers $Y_i = \{1,2,\hdots,t-1,t+i-1\}$, $1 \le i \le r$, and their union has size $s = \binom{n-t+1}{k-t+1} - \binom{n-t-r+1}{k-t+1}$.  Letting $\cL = \cL_{n,k}\left( \binom{n-t+1}{k-t+1} - \binom{n-t-r+1}{k-t+1} \right)$, note that for any set $I \subset [r]$, since $|\cup_{i \in I} Y_i| = t + |I| -1$, we have $| \cap_{i \in I} \cL(Y_i) | = \binom{n-t-|I|+1}{k - t - |I| + 1}$.  Thus, by Inclusion-Exclusion,
\begin{align*}
 |\cL| = |\cup_{i=1}^r \cL(Y_i)| &= \sum_i |\cL(Y_i)| - \sum_{i_1 < i_2} |\cL(Y_{i_1}) \cap \cL(Y_{i_2})| + O(n^{k-t-2}) \\
  &= r \binom{n-t}{k-t} - \binom{r}{2} \binom{n-t-1}{k-t-1} + O(n^{k-t-2}).
\end{align*}

Now we consider the size of $\cF$.  Suppose $\cF$ is the union of the $r$ full $t$-stars with centers $\{X_1, \hdots, X_r \}$.  We have
\[ |\cF| = |\cup_{i=1}^r \cF(X_i)| \ge \sum_{i=1}^r |\cF(X_i)| - \sum_{i_1 < i_2} | \cF(X_{i_1}) \cap \cF(X_{i_2})| = r\binom{n-t}{k-t} - \sum_{i_1 < i_2} | \cF(X_{i_1}) \cap \cF(X_{i_2})|. \]

For every $i_1 < i_2$ we have $|\cF(X_{i_1}) \cap \cF(X_{i_2})| = \binom{n-|X_{i_1} \cup X_{i_2}|}{k - |X_{i_1} \cup X_{i_2}|}$.  If $|X_{i_1} \cap X_{i_2}| \le t-2$, then $\card{X_{i_1} \cup X_{i_2}} \ge t + 2$.  Hence $|\cF(X_{i_1}) \cap \cF(X_{i_2})| = O(n^{k-t-2})$, and so
\[ |\cF| \ge r\binom{n-t}{k-t} - \left( \binom{r}{2} - 1 \right) \binom{n-t-1}{k-t-1} + O(n^{k-t-2}) > |\cL|. \]
Hence we must have $|X_{i_1} \cap X_{i_2}| = t-1$ for all $i_1 < i_2$.

Now, by Inclusion-Exclusion, we have
\[ |\cF| - r\binom{n-t}{k-t} + \binom{r}{2} \binom{n-t-1}{k-t-1} = \sum_{\substack{I \subset [r] \\ |I| \ge 3}} (-1)^{|I|+1} | \cap_{i \in I} \cF(X_i) |. \]

For any set $F$ containing $a \ge 3$ sets $X_i$, the contribution to the right-hand side is 
\[ \sum_{b=3}^a (-1)^{b+1} \binom{a}{b} = (1-1)^a + 1 - a + \binom{a}{2} = 1 - a + \binom{a}{2} \ge 1. \]
If we have some $i_1 < i_2 < i_3$ with $|X_{i_1} \cup X_{i_2} \cup X_{i_3}| = t+1$, then we would have $\binom{n-t-1}{k-t-1}$ sets containing $X_{i_1}$, $X_{i_2}$ and $X_{i_3}$.  By the preceding equation, we then have
\[ |\cF| \ge r\binom{n-t}{k-t} - \binom{r}{2} \binom{n-t-1}{k-t-1} + \binom{n-t-1}{k-t-1} > |\cL|. \]

Hence we may assume $|X_{i_1} \cup X_{i_2} \cup X_{i_3}| \ge t+2$ for all $i_1 < i_2 < i_3$.  Since we must have $|X_{i_1} \cap X_{i_2}| = t-1$ for all $i_1 < i_2$, this implies all of the sets $X_i$ share a common $(t-1)$-set, and hence $\cF$ is isomorphic to $\cL$, as desired.
\end{proof}

The next lemma showed that when adding a set to $r$ full $t$-stars, the lexicographical stars minimize the number of new $t$-disjoint pairs.

\begin{proof}[Proof of Lemma \ref{lem:addset}]
$\cL$ is the union of the $t$-stars with centers $\{Y_1, Y_2, \hdots, Y_r\}$, as in Lemma \ref{lem:lexsmall}.  Since all these sets, and $L$, contain $[t-1]$, it is easy to see that
\begin{align*}
 \disj_t(L,\cL) &\le \sum_{i=1}^r \disj_t(L,\cL(Y_i)) - \sum_{i_1 < i_2} \disj_t(L, \cL(Y_{i_1}) \cap \cL(Y_{i_2})) \\
    & \quad + \sum_{i_1 < i_2 < i_3} \disj_t(L,\cL(Y_{i_1}) \cap \cL(Y_{i_2}) \cap \cL(Y_{i_3})) \\
    &= r\binom{n-k-1}{k-t} - \binom{r}{2} \binom{n-k-2}{k-t-1} + O(n^{k-t-2}).
\end{align*}

On the other hand, we have
\[ \disj_t(F, \cF) \ge \sum_{i=1}^r \disj_t(F, \cF(X_i)) - \sum_{i_1 < i_2} \disj_t(F, \cF(X_{i_1}) \cap \cF(X_{i_2})). \]

The first term can be evaluated as follows.  Since
\[ \disj_t(F,\cF(X_i)) = \sum_{a=0}^{t-1-|F \cap X_i|} \binom{k-|F \cap X_i|}{a} \binom{n - k - t + |F \cap X_i|}{k - t - a}, \]
if $|F \cap X_i| = t-1$ we have $\disj_t(F, \cF(X_i)) = \binom{n-k-1}{k-t}$, while $\disj_t(F, \cF(X_i)) \ge \binom{n-k-2}{k-t} + (k-t+2)\binom{n-k-2}{k-t-1} = \binom{n-k-1}{k-t} + (k-t+1)\binom{n-k-2}{k-t-1}$ otherwise.  Moreover, for every $i_1 < i_2$ we have the bound $\disj_t(F, \cF(X_{i_1}) \cap \cF(X_{i_2})) \le | \cF(X_{i_1}) \cap \cF(X_{i_2})| \le \binom{n-t-1}{k-t-1}$.  Hence, if $|F \cap X_i| \le t-2$ for some $i$,
\begin{align*}
 \disj_t(F, \cF) &\ge r\binom{n-k-1}{k-t} + (k-t+1)\binom{n-k-2}{k-t-1} - \binom{r}{2} \binom{n-t-1}{k-t-1} \\
 &= r \binom{n-k-1}{k-t} - \left( \binom{r}{2} - (k-t+1) \right) \binom{n-k-2}{k-t-1} + O(n^{k-t-2}) > \disj_t(L,\cL).
\end{align*}

Thus we may assume $|F \cap X_i| = t-1$ for all $i$.  Given this condition, it follows that
\[ \disj_t(F,\cF(X_{i_1}) \cap \cF(X_{i_2})) = \left\{ \begin{array}{cl}
    \binom{n-k-1}{k-t-1} & \textrm{if } F \cap X_{i_1} = F \cap X_{i_2} \\
    0 & \textrm{otherwise, since } \card{ F \cap \left( X_{i_1} \cup X_{i_2} \right) } \ge t
    \end{array} \right. . \]

Hence, in order to have $\disj_t(F,\cF) \le \disj_t(L, \cL) = r\binom{n-k-1}{k-t} - \binom{r}{2} \binom{n-k-2}{k-t-1} + O(n^{k-t-2})$, we must have $F \cap X_{i_1} = F \cap X_{i_2}$ for all $i_1 < i_2$.  This implies that $F$ shares a common $(t-1)$-set with all the sets $X_i$, and thus $\cF \cup \{F\}$ is isomorphic to $\cL \cup \{L\}$, as required.
\end{proof}

The final lemma showed that the union of any $r$ full $t$-stars contains at least as many disjoint pairs as the initial segment of the lexicographical ordering with the same number of sets.

\begin{proof}[Proof of Lemma \ref{lem:fullstars}]
We shall find it more convenient to count the number of $t$-intersecting pairs.  Suppose $\cF$ is the union of the full $t$-stars with centers $\{ X_1, X_2, \hdots, X_r \} \subset \binom{[n]}{t}$.  By Lemma \ref{lem:lexsmall}, it follows that $\cL = \cL_{n,k}(\card{\cF})$ consists of the full $t$-stars with centers $\{Y_1, Y_2, \hdots, Y_r \}$, possibly with some additional sets in an $(r+1)$st $t$-star with center $Y_{r+1}$, where $Y_i = \{1,2, \hdots, t-1, t-1+i\}$.  Note that in this setting we have $\card{\cF} = \card{\cL}$.

\medskip

We first show that if $\card{X_i \cap X_j} \le t-2$ for some $1 \le i < j \le r$, then the $r$ full $t$-stars of $\cL$ alone contain more $t$-intersecting pairs than $\cF$.  We have
\begin{align} \label{ineq:tintbound}
	\tint(\cF) &\le \sum_{i=1}^r \tint(\cF(X_i)) + \sum_{i < j} \tint(\cF(X_i) \setminus \cF(X_j), \cF(X_j) \setminus \cF(X_i)) \notag \\
	&= r \binom{ \binom{n-t}{k-t} }{2} + r \binom{n-t}{k-t} + \sum_{i < j} \tint(\cF(X_i) \setminus \cF(X_j), \cF(X_j) \setminus \cF(X_i)),
\end{align}
where the inequality is due to the fact that $t$-intersecting pairs involving sets in multiple $t$-stars are overcounted.

\medskip

First suppose $\card{X_i \cap X_j} = t-1$.  Given a set $F \in \cF(X_i) \setminus \cF(X_j)$, we wish to bound how many sets $G \in \cF(X_j) \setminus \cF(X_i)$ can be $t$-intersecting with $F$.  Since $X_i \cap X_j \subset F \cap G$, we require $G$ to contain one additional element of $F$.  However, this element cannot be from $X_i$, as then we would have $G \in \cF(X_i)$.  Thus there are $k-t$ choices for this additional element.  Given that $G$ already contains $X_j$, there are $\binom{n-t-1}{k-t-1}$ ways to choose the remaining elements of $G$.  Hence there can be at most $(k-t) \binom{n-t-1}{k-t-1}$ such sets $G$, giving
\[ \tint(\cF(X_i) \setminus \cF(X_j), \cF(X_j) \setminus \cF(X_i)) \le (k-t) \binom{n-t-1}{k-t-1} \card{ \cF(X_i) \setminus \cF(X_j) } \le (k-t) \binom{n-t}{k-t} \binom{n-t-1}{k-t-1}. \]

Now suppose $\card{X_i \cap X_j} \le t-2$.  There are two types of $F \in \cF(X_i) \setminus \cF(X_j)$: those with $F \cap X_j = X_i \cap X_j$, and those with $(F \setminus X_i) \cap X_j \neq \emptyset$.  In the first case, note that for $G \in \cF(X_j) \setminus \cF(X_i)$ to be $t$-intersecting with $F$, $G$ must contain at least $2$ elements from $F$ in addition to $X_j$.  Hence there are at most $\binom{k}{2} \binom{n-t-2}{k-t-2}$ such sets.  In the second case, note that there are at most $t \binom{n-t-1}{k-t-1}$ such sets $F$, as we can choose at most $t$ elements from $X_j \setminus X_i$ for $F$ to contain, and then there are $\binom{n-t-1}{k-t-1}$ ways to choose the remaining elements for $F$.  For each such $F$, in order for $G \in \cF(X_j) \setminus \cF(X_i)$ to be $t$-intersecting with $F$, $G$ must contain some element of $F$ in addition to $F \cap X_j$.  There are at most $k$ choices for this element, with $\binom{n-t-1}{k-t-1}$ ways to complete $G$.  Thus there are at most $kt \binom{n-t-1}{k-t-1}^2$ $t$-intersections of this type.  Hence we have
\[ \tint(\cF(X_i) \setminus \cF(X_j) , \cF(X_j) \setminus \cF(X_i)) \le \binom{k}{2} \binom{n-t}{k-t} \binom{n-t-2}{k-t-2} + kt \binom{n-t-1}{k-t-1}^2 = O(n^{2k - 2t - 2}). \]

Substituting these bounds into \eqref{ineq:tintbound}, if we have $d$ pairs $\{ i, j \}$ with $\card{X_i \cap X_j} \le t-2$, we have
\begin{equation} \label{ineq:tintbound2}
	\tint(\cF) \le r \binom{\binom{n-t}{k-t} }{2} + \left( \binom{r}{2} -d \right) (k-t) \binom{n-t}{k-t} \binom{n-t-1}{k-t-1} + O(n^{2k-2t-2}).
\end{equation}

We now provide a lower bound for $\tint(\cL)$, considering only the $r$ full $t$-stars.  There are two types of $t$-intersecting pairs: those from within a single $t$-star, and those between two $t$-stars.

\medskip

For the first kind, note that there are $r$ $t$-stars, with $\binom{ \binom{n-t}{k-t} }{2} + \binom{n-t}{k-t}$ $t$-intersecting pairs in each.  However, this overcounts those pairs that are contained in the intersection of multiple $t$-stars.  Any two $t$-stars in $\cL$ share $\binom{n-t-1}{k-t-1}$ sets, and so we overcount at most $\binom{r}{2} \left[ \binom{ \binom{n-t-1}{k-t-1} }{2} + \binom{n-t-1}{k-t-1} \right]$ pairs.  Hence the number of $t$-intersecting pairs within $t$-stars is at least
\[ r \left[ \binom{ \binom{n-t}{k-t} }{2} + \binom{n-t}{k-t} \right] - \binom{r}{2} \left[ \binom{ \binom{n-t-1}{k-t-1} }{2} + \binom{n-t-1}{k-t-1} \right]= r \binom{ \binom{n-t}{k-t} }{2} + O(n^{2k-2t-2}). \]

For the second kind, we wish to avoid counting those pairs that already appear within some star.  This can only happen if one of the sets is in multiple $t$-stars, and so we define $\cL^-(Y_i)$ to be those sets in $\cL$ that are only in the $t$-star $\cL(Y_i)$.  The number of such sets is $\binom{n-t-r+1}{k-t} = \binom{n-t}{k-t} + O(n^{k-t-1})$, as they must contain $Y_i$, and avoid the $r-1$ elements in $\cup_j Y_j \setminus Y_i$.

Since $Y_i \cap Y_j = [t-1]$, a set $G \in \cL^-(Y_j)$ is $t$-intersecting with a set $F \in \cL^-(Y_i)$ if it contains an element in $F \setminus Y_i$.  To avoid double-counting, we shall only consider those $G \in \cL^-(Y_j)$ that meet $F \setminus Y_i$ exactly once.  Moreover, since we require $G \in \cL^-(Y_j)$, $G$ should avoid the $r-1$ elements in $\cup_a Y_a \setminus Y_j$.  This gives $\tint(F, \cL^-(Y_j)) \ge (k-t) \binom{n-k-r}{k-t-1}$.  Hence we obtain
\begin{align*}
	\tint(\cL^*(Y_i), \cL^*(Y_j)) &\ge (k-t) \binom{n-t-r+1}{k-t} \binom{n-k-r}{k-t-1} \\
	&= (k-t) \binom{n-t}{k-t} \binom{ n - t - 1} {k-t-1} + O(n^{2k - 2t - 2}).
\end{align*}

Thus we deduce
\[ \tint(\cL) \ge r \binom{ \binom{n-t}{k-t} }{2} + \binom{r}{2} (k-t) \binom{n-t}{k-t} \binom{n-t-1}{k-t-1} + O(n^{2k - 2t - 2}).\]

Comparing this to \eqref{ineq:tintbound2}, we find that unless $d = 0$, we must have $\tint(\cL) > \tint(\cF)$, as desired.  It remains to consider the case when $\card{X_i \cap X_j} = t-1$ for all $1 \le i < j \le r$.

\medskip

There are only two possibilities.  In the first, all the sets $X_i$ share $t-1$ elements in common, in which case $\cF$ is isomorphic to $\cL$.  The second case, up to isomorphism, is when $r \le t+1$, and $X_i \in \binom{[t+1]}{t}$.  Note that if $1 \le r \le 2$, the two constructions are isomorphic, so we may assume $r \ge 3$.

In this case, as we know the exact structure of both constructions, we are able to compute the number of intersecting pairs rather more precisely.  We begin with $\cF$, the union of $r$ full $t$-stars with centers from $\binom{[t+1]}{t}$.

\medskip

$\cF$ contains all $\binom{n-t-1}{k-t-1}$ sets containing $[t+1]$, and then $r \binom{n-t-1}{k-t}$ sets that meet $[t+1]$ in $t$ elements.  The sets containing $[t+1]$ are $t$-intersecting with all other sets in $\cF$.

On the other hand, if $F \in \cF(X_i)$ is such that $F \cap [t+1] = X_i$, then there are three types of sets in $\cF$ that can be $t$-intersecting with $F$:
\begin{itemize}
	\item[(i)] a set containing $[t+1]$,
	\item[(ii)] a set whose intersection with $[t+1]$ is precisely $X_i$, or
	\item[(iii)] a set whose intersection with $[t+1]$ is $X_j$ for some $j \neq i$.
\end{itemize}
There are $\binom{n-t-1}{k-t-1}$ sets of type $(i)$ and $\binom{n-t-1}{k-t}$ sets of type $(ii)$.  For a set to be of type $(iii)$, it must contain some $X_j$, not contain $X_i$, and then meet $F$ in some element of $F \setminus X_i$.  For each choice of $j$, the set should contain the $k$ elements of $X_j$, not the single element in $X_i \setminus X_j$, and should not avoid the remaining $k - t$ elements of $F$.  Hence there are $\binom{n-t-1}{k-t} - \binom{n-k-1}{k-t}$ such sets.

Putting this all together, we find
\begin{equation} \label{eqn:tclique}
	2 \tint(\cF) - \card{\cF} = \sum_{F \in \cF} \tint(F, \cF) = I_1 + I_2,
\end{equation}
where
\begin{align*}
I_1 &= \binom{n-t-1}{k-t-1} \left[ \binom{n-t-1}{k-t-1} + r \binom{n-t-1}{k-t} \right], \textrm{ and } \\
I_2 &= r \binom{n-t-1}{k-t} \left[ \binom{n-t-1}{k-t-1} + \binom{n-t-1}{k-t} + (r-1) \left( \binom{n-t-1}{k-t} - \binom{n-k-1}{k-t} \right) \right].
\end{align*}

We now turn our attention to $\cL$.

\medskip

First observe that we have $r$ full stars, with centers $\{Y_1, Y_2, \hdots, Y_r\}$.  The remaining sets fall into an $(r+1)$st star with center $Y_{r+1}$.  To avoid overcounting, we shall partition $\cL$ into the subsystems $\cL^*(i) = \{L \in \cL: \min (L \setminus [t-1]) = t-1+i \}$, $1 \le i \le r+1$; that is, $L \in \cL^*(i)$ if $\cL(Y_i)$ is the first $t$-star $L$ is in.

For $1 \le i \le r$, $\cL^*(i)$ consists of all sets containing $[t-1] \cup \{t-1+i\}$, but disjoint from the interval $[t,t-2+i]$.  Hence we have $\card{\cL^*(i)} = \binom{n-t-i+1}{k-t}$.  Summing up the telescoping binomial coefficients, we find the first $r$ $t$-stars contain $\binom{n-t+1}{k-t+1} - \binom{n-t-r+1}{k-t+1}$ sets.  $\cL^*(r+1)$ then contains enough sets to make $\cL$ the right size, and so
\[ \card{\cL^*(r+1)} = \card{\cF} - \card{\cup_{i=1}^r \cL^*(i)} = \left[ \binom{n-t-1}{k-t-1} + r\binom{n-t-1}{k-t} \right] - \left[ \binom{n-t+1}{k-t+1} - \binom{n-t-r+1}{k-t+1} \right]. \]

Note that all the subsystems $\cL^*(i)$ are $t$-intersecting.  Moreover, if $j < i$, and $L \in \cL^*(i)$, then for a set $K \in \cL^*(j)$ to be $t$-intersecting with $L$, it must contain $[t-1] \cup \{t-1+j\}$, be disjoint from the interval $[t,t-2+j]$, and contain one of the $k-t+1$ elements in $L \setminus [t-1+j]$.  Hence we have $\tint(L,\cL^*(j)) = \binom{n-t-j+1}{k-t} - \binom{n-k-j}{k-t}$.  We can now count the number of $t$-intersecting pairs in $\cL$:
\begin{align} \label{eqn:tlex}
	2 \tint(\cL) - \card{\cL} &= \sum_{L \in \cL} \tint(L, \cL) = \sum_{i=1}^{r+1} \left( \tint(\cL^*(i), \cL^*(i)) + 2 \sum_{j < i} \tint(\cL^*(i), \cL^*(j)) \right) \notag \\
	&= \sum_{i=1}^{r+1} \card{\cL^*(i)}^2 + 2 \sum_{j < i} \card{\cL^*(i)} \left[ \binom{n-t-j+1}{k-t} - \binom{n-k-j}{k-t} \right].
\end{align}

We now wish to show $\tint(\cL) \ge \tint(\cF)$; that is, to show the quantity in \eqref{eqn:tlex} is greater than that in \eqref{eqn:tclique}.  To make this task easier, we shall rewrite all products of binomial coefficients in the form $\binom{n-t}{k-t}^2, \binom{n-t}{k-t} \binom{n-t}{k-t-1},$ or $\binom{n-t}{k-t-1}^2$, using the identities
\begin{align*}
	\binom{m-a}{r} &= \binom{m}{r} - a \binom{m}{r-1} + \binom{a+1}{2} \binom{m}{r-2} + O(m^{r-3}) \textrm{ and} \\
	\binom{n-t}{k-t} \binom{n-t}{k-t-2} &= \frac{n-k+1}{n-k+2} \cdot \frac{k-t-1}{k-t} \cdot \binom{n-t}{k-t-1}^2 = \frac{k-t-1}{k-t} \binom{n-t}{k-t-1}^2 + O(n^{2k-2t-3}).
\end{align*}

After performing the routine but tedious calculations, we find, up to an error of $O(n^{2k - 2t - 3})$,
\begin{align*}
	2 \tint( \cF ) - \card{\cF} &= r \binom{n-t}{k-t}^2 + r(r-1)(k-t) \binom{n-t}{k-t} \binom{n-t}{k-t-1} \\
	&\quad - \left[ \frac12 r(r-1)(k-t)^2 + 2r(r-1)(k-t) - \left( \frac{3r}{2} - 1 \right) \left( r - 1 \right) \right] \binom{n-t}{k-t-1}^2, \\
\textrm{ and } 2 \tint( \cL ) - \card{\cL} &= r \binom{n-t}{k-t}^2 + r(r-1)(k-t) \binom{n-t}{k-t} \binom{n-t}{k-t-1} \\
	&\quad - \left[ \frac12 r(r-1)(k-t)^2 + 2r(r-1)(k-t) - \frac14 (r-1)^2 (r^2+4) \right] \binom{n-t}{k-t-1}^2.
\end{align*}

The coefficient of the leading term of the difference between the two constructions is thus
\begin{align*}
	\frac{2 \tint( \cL ) - 2 \tint( \cF )}{\binom{n-t}{k-t-1}^2} &= \left( \frac14 (r-1)^2 (r^2 + 4) - \left(\frac{3r}{2} - 1 \right)(r - 1) \right) \\
	&= \frac14 (r+1)r(r-1)(r-2) > 0
\end{align*}
for $r \ge 3$.  Hence we indeed find $\tint(\cL) > \tint(\cF)$, as required.
\end{proof}

\section{Concluding remarks and open problems} \label{sec:conclusion}

In this paper, we have provided a partial solution to a problem of Ahlswede on the minimum number of disjoint pairs in set systems.  For small systems, we verified Bollob\'as and Leader's conjecture by showing that the initial segment of the lexicographical ordering is optimal.  By considering the complementary set systems, this also resolves the problem for very large set systems.  However, it remains to determine which systems are optimal in between.

When $k=2$, Ahlswede and Katona showed that the optimal system was always either a union of stars or its complement.  For $k \ge 3$, Bollob\'as and Leader suggest a larger family of possible extremal systems.  We note that for systems of size $s = \frac12 \binom{n}{k}$, the lexicographical system is at least near-optimal.  A straightforward calculation shows $\disj(n,k,s) \le \disj(\cL_{n,k}(s)) \le \frac12 \left( 1 - \frac{2^{1 / k} k^2}{n} + O(n^{-2}) \right) s^2$.  On the other hand, exploiting the connection to the Kneser graph, we can use spectral techniques to obtain the bound $\disj(n,k,s) \ge \frac12 \left( 1 - \frac{k(k+2)}{n} + O(n^{-2}) \right) s^2$.

\medskip

While our focus has been showing that a system with more than $\binom{n-1}{k-1}$ sets must contain many disjoint pairs, a closely related problem is to determine whether such a system must have any sets disjoint from many other sets.  This type of question has been studied before in other settings.  For example, when one is considering the number of triangles in a graph, Erd\H{o}s showed in \cite{erdos62a} that any graph with $\floor{ \frac{n^2}{4} } + 1$ edges must contain an edge in at least $\frac{n}{6} + o(n)$ triangles. It is well-known and easy to see that the hypercube, a graph whose vertices are subsets of $[n]$, with two vertices adjacent if they are comparable and differ in exactly one element, has independence number $2^{n-1}$.  In \cite{chung88}, it is proved that any induced subgraph on $2^{n-1} + 1$ vertices contains a vertex of degree at least $(\frac12 + o(1)) \log_2 n$.  It is an open problem to determine whether or not this bound is tight (the corresponding upper bound is $O(\sqrt{n})$), and the answer to this question has ramifications in theoretical computer science.

In the context of the Erd\H{o}s-Ko-Rado Theorem, it is trivial to show that in a system of $\binom{n-1}{k-1} + 1$ sets, there must be a set disjoint from at least $\frac12 \left( 1 - \frac{k^3}{n} \right) \binom{n-1}{k-1}$ other sets.  Indeed, by the Erd\H{o}s-Ko-Rado theorem, there exists a pair $F_1, F_2$ of disjoint sets.  At most $k^2 \binom{n-2}{k-2} < \frac{k^3}{n} \binom{n-1}{k-1}$ sets can intersect both $F_1$ and $F_2$, and so either $F_1$ or $F_2$ must be disjoint from at least half of the remaining sets, resulting in the above bound.  Furthermore, this is easily seen to be asymptotically tight, as one may take all sets containing $\{1,2\}$, and then take half the remaining sets to contain $1$, and half to contain $2$.  It may be of interest to obtain sharper estimates for this problem, especially as the aforementioned construction shows that this is closely related to the original problem when $s \approx \frac12 \binom{n}{k}$, since one should choose the sets containing $1$ or $2$ optimally.

\medskip

We find most exciting the prospect of studying Erd\H{o}s-Rademacher-type problems in other settings.  In an earlier paper, we presented an Erd\H{o}s-Rademacher-type strengthening of Sperner's Theorem, a problem that was also studied in \cite{dove13}.  However, as one can investigate similar extensions for any extremal result, there is truly no end to the number of directions in which this project can be continued.  We hope that further work of this nature will lead to many interesting results and a greater understanding of classical theorems in extremal combinatorics.

\end{document}